\documentclass{article}
\usepackage[utf8]{inputenc}
\usepackage{amsfonts}
\usepackage{amsmath}
\usepackage{amsthm}
\usepackage{enumitem}
\usepackage{setspace}
\usepackage{algorithm}
\usepackage{algpseudocode}
\usepackage{amssymb}
\usepackage{stmaryrd}
\usepackage{makecell}
\usepackage{caption}
\usepackage{url}
\usepackage[T1]{fontenc}
\usepackage{tikz-cd}
\usepackage{siunitx}

\newcommand{\bigo}{\mathcal{O}}
\newcommand{\tbigo}{\widetilde \bigo}

\newcommand{\Z}{\mathbb{Z}}
\newcommand{\F}{\mathbb{F}}

\newcommand{\legendre}[2]{\genfrac{(}{)}{}{}{#1}{#2}}
\newcommand{\generated}[1]{\ensuremath{\langle #1\rangle}}

\newtheorem{theorem}{Theorem}[section]
\newtheorem{prop}[theorem]{Proposition}
\newtheorem{lemma}[theorem]{Lemma}
\newtheorem{cor}[theorem]{Corollary}
\newtheorem{example}[theorem]{Example}
\newtheorem{remark}[theorem]{Remark}
\newtheorem{definition}[theorem]{Definition}
\newtheorem{assumptions}[theorem]{Assumptions}
\newtheorem*{acknowledgements*}{Acknowledgements}

\numberwithin{equation}{section}
\title{Fast square-free decomposition of integers using class groups}
\author{Erik Mulder}
\date{July 2023}

\begin{document}

\maketitle

\begin{abstract}
    Let $n=a^2b$, where $b$ is square-free. In this paper we present an algorithm based on class groups of binary quadratic forms that finds the square-free decomposition of $n$, i.e. $a$ and $b$, in heuristic expected time:
    \[ \tbigo{(L_{b}[1/2,1] \ln(n) + L_{b}[1/2,1/2] \ln(n)^2)}. \]
    If $a,b$ are both primes of roughly the same cryptographic size, then our method is currently the fastest known method to factor $n$.
    This has applications in cryptography, since some cryptosystems rely on the hardness of factoring integers of this form.
\end{abstract}

\newpage
\tableofcontents
\newpage

\section{Introduction}

% TO DO:
% \begin{itemize}
%     \item Ik zeg steeds `our algorithm', misschien een (afgekorte) naam eraan geven? Ook daarna de titels van de algorithms aanpassen (en alle verwijzingen).
%     \item Formatting of tables?
%     \item Lattices idee? Bouwt eigenlijk voort op FFT verhaal over kleine square root.
%     Suggereert een verband tussen $\sqrt{f[1]}$ en $\sqrt{f^2[1]} \bmod r^2$, wat eigenlijk wel erg interessant is.
%     \item Referentielijst met of zonder voornamen
%     \item Meer examples toevoegen?
%     \item Maakt het plaatje in Lemma \ref{r truc} het helderder of juist niet? Kan ook nog verplaatst worden.
%     \item Het feit dat Appendix \ref{appendix: composite square factor} los staat is nog steeds een beetje ongemakkelijk
%     \item Theorem \ref{lenstra refrase} is ook een beetje ongemakkelijk
% \end{itemize}

% \newpage

\noindent One of the classic questions in computational number theory is whether the square-free decomposition of a given mathematical object can be computed `fast'.
The question for polynomials is easy to solve: given a polynomial $f = f_1f_2^2$, where $f_1$ is square-free, we can compute $g = \gcd(f,f')$. Then $g$ is a multiple of $f_2$, and with some more care, the exact square factor $f_2$ can be found in polynomial time \cite{yun_square_polynomials}.
For integers however, the question is still open. In this paper we present a novel technique that uses class groups to find the square-free decomposition of an integer $n = a^2b$, with $b$ square-free.
The algorithm runs in heuristic expected time 
\[ \tbigo{(L_{b}[1/2,1] \ln(n) + L_{b}[1/2,1/2] \ln(n)^2)} \]
where $L_b$ is the $L$-notation:
%\[ \tbigo{(L_b[1/2,1] \ln(n))} = \tbigo{(e^{\sqrt{(\ln(b) \ln\ln(b))}} \ln(n)}) \]
%\[ \tbigo{(L_b[1/2,1] \ln(n))} = \tbigo{(e^{(1 + o(1)) \sqrt{\ln(b) \ln\ln(b)}}} \ln(n)) \]
\[ L_b[\alpha,c] = e^{(c + o(1)) \ln(b)^{\alpha} (\ln\ln(b))^{1-\alpha}} \]

% Given $n = a^2b$, our algorithm takes a random form $f$ in the class group $C(-4n)$ and then computes $g = f^k$, where $k$ is a large highly composite integer.
% If $h(-4n)$ is smooth, then $g \sim e$, where $e$ is the identity element of $C(-4n)$.
% Schnorr and Lenstra showed \cite{schnorrlenstra} how to completely factor $n$ in this case.
% We will use the fact that every form in $g \in C(-4n)$ can be derived from a unique form $\pi(g) \in C(-4b)$.
% We don't require that $g \sim e \in C(-4n)$, but instead that $\pi(g) \sim e \in C(-4b)$.
% Given such $g$, we show how to retrieve $a$, which gives us the square-free decomposition of $n$.
% Our method is successful if the class number $h(-4b)$ is smooth. If $b$ is not too big, then the probability that this happens is much larger than the probability that $h(-4n)$ is smooth. 

In the special case that $a=p$ and $b=q$ are distinct primes and $p \approx q$, our algorithm is currently the fastest known method for computing the square-free decomposition if $q$ is roughly in the range $[10^{20}, 10^{5000}]$. The upper bound should be taken with a grain of salt, see Section \ref{ss: complexity comparisons}.
Numbers $n$ of this form might seem like a very specific case.
It is true that the probability that a random large integer $n$ is of this form is very low.
However, these numbers appear quite frequently in cryptographic systems \cite{crypto_okamoto_1998, crypto_okamoto_1990, crypto_paulus, crypto_schmidt, crypto_takagi_rsa}.
In these cryptosystems, the assumption is usually made that factoring an integer $n =  p^2q$ with $p \approx q $ is as hard as factoring an arbitrary integer with 3 large prime factors. Our method shows that this is not true when $q$ is of cryptographic size. Therefore, larger moduli should be used in these cryptosystems than advised in those articles.

Another application of our algorithm is to determine the ring of integers of number fields, since this is polynomial time equivalent to finding the square-free decomposition of the discriminant of the number field \cite{chistov_ring_of_integers}.
In \cite{lenstra_buchmann_ring_of_integers}, this equivalence is studied in great detail.
Unfortunately, our algorithm does not run in polynomial time, but it can be useful when the square-free part of the discriminant is not too big.
Another possible application is in determining the endomorphism ring of an elliptic curve over a finite field \cite{endomorphism_ring}.
Since that algorithm requires the square part of the discriminant of the characteristic equation of the Frobenius endomorphism.

Factoring integers that have repeated prime factors has been an area of active research for quite some time.
Lattice algorithms are quite popular \cite{boneh_factoring_prq, coron_factoring_prqs, may_factoring_p^rq, harvey_factoring_p^rq}, but also elliptic curves \cite{peralta_ecm_p2q} and even class groups \cite{Factoring_pq2, truc_met_r} have been used before.
In \cite{detectingsquarefree} an algorithm is presented that also uses class groups, which can be used to detect square-free numbers.
Finally, it is good to mention that general purpose factorization algorithms such as the number field sieve \cite{lenstra_nfs} or the elliptic curve method \cite{lenstraECM} can of course also factor numbers of this form.

In 1984, Schnorr and Lenstra \cite{schnorrlenstra} presented the following algorithm to factor an integer $n$, which is very similar to other factorization algorithms such as the elliptic curve method and Pollard's $p-1$ method.
Take a random form $f$ in the class group $C(-4n)$ and compute $g = f^k$, where $k$ is a large highly composite integer.
If $h(-4n)$ is smooth, then $g \sim e$, where $e$ is the identity element of $C(-4n)$.
They showed how to factor $n$ in this case.
Based on heuristic assumptions, they claimed that the expected runtime of their algorithm is $\bigo(L_n[1/2,1])$.
Unfortunately, it was later found that if $n$ has a large square prime divisor, then this runtime is out of reach \cite{lenstra_flaw}.

We adapt the algorithm of Schnorr and Lenstra such that it actually becomes \emph{faster} when $n = a^2b$ has a large square divisor.
We will use the fact that every form in $g \in C(-4n)$ can be derived from a unique form $\pi(g) \in C(-4b)$.
We don't require that $g \sim e \in C(-4n)$, but instead that $\pi(g) \sim e \in C(-4b)$.
Given such $g$, we show how to retrieve $a$, which gives the square-free decomposition of $n$.
Our method is successful if the class number $h(-4b)$ is smooth. If $b$ is not too big, then the probability that this happens is much larger than the probability that $h(-4n)$ is smooth.

\subsection{Outline of this paper}
In Chapter 2 we will recall some basic properties of binary quadratic forms and how forms in $C(-4a^2b)$ can be derived from forms in $C(-4b)$.
In Chapters 3 and 4, we look at the factorization algorithm of Schnorr and Lenstra \cite{schnorrlenstra} and we adapt it to make it fast for integers of the form $n = p^2q$.
In Chapter 5 we compare our algorithm to other factorization algorithms and we test its speed in practice.

This article also has four appendices.
In Appendix \ref{appendix: composite square factor} we show that our algorithm works more generally for integers $n = a^2b$, where $a$ and $b$ can both be composite.
Appendices \ref{appendix: multipliers} and \ref{appendix: stage 2} are more experimental in nature, there we try to find adjustments that speed up our algorithm.
In Appendix \ref{appendix: implementation} we provide a basic implementation of our algorithm in \textsc{Magma} \cite{magma}.

\begin{acknowledgements*}
    \normalfont
    The author would like to thank Wieb Bosma for the helpful discussions and for proofreading this paper.
\end{acknowledgements*}

\section{Class group of binary quadratic forms}

\subsection{Preliminaries}
In this section we quickly recall important definitions from the theory of binary quadratic forms. For more details see Cox \cite{cox}, Chapters 1 to 3.

A \textit{binary quadratic form} $f$ is a polynomial of the form $f(x,y) = ax^2+bxy+cy^2 = (a,b,c)$, where $a,b,c \in \Z$.
The \textit{discriminant} of $f$ is $D = b^2-4ac = 0,1 \bmod 4$.
We say that $f$ \textit{represents} $m$ if there exist $x,y \in \Z$ such that $f(x,y) = m$.
A form is \textit{primitive} if $\gcd(a,b,c) = 1$. We will always assume that our forms are primitive, unless stated otherwise.
Let $\Gamma$ be the classical modular group of $2 \times 2$ matrices $A$ with integer coefficients
and $\det(A) = 1$. Two binary quadratic forms $f,g$ are \textit{equivalent} if there exists a matrix
$A =\left(\begin{smallmatrix}
    p & q\\r & s
  \end{smallmatrix}\right)
  \in \Gamma$
such that $f(x,y) = g(A \cdot (x,y)^T) = g(px+qy, rx+sy)$. If this is the case, then we write $f \sim g$. This implies that equivalent forms have the same discriminant.
A form $(a,b,c)$ of discriminant $D < 0$ is \textit{reduced} if $|b| \leq a \leq c$, and $b \geq 0$ if either $|b| = a$ or $a = c$.
Every form of negative discriminant is equivalent to a unique reduced form.

The \textit{class group} $C(D)$ consists of the equivalence classes of primitive binary quadratic forms of discriminant $D$, the group operation is composition of forms.
Given two forms $f = (a_1,b_1,c_2)$, $g = (a_2,b_2,c_2)$ in $C(D)$ with $\gcd(a_1,a_2,(b_1+b_2)/2) = 1$, their (Dirichlet) composition is $f \cdot g 
= (a_1a_2, B, \frac{B^2 - D}{4a_1a_2})$,
for a suitable integer $B$.
%where $B = b_1 \bmod 2a_1$ and $B = b_2 \bmod 2a_2$ and $B^2 = D \bmod 4a_1a_2$.
For more details about this operation, see Chapter 3 of \cite{cox}.
This composition can be computed in $\tbigo(\ln(D))$ using a variant of fast GCD \cite{schonhagefastcomposition}.
The \textit{class number} $h(D)$ is the order of $C(D)$.
For $D < 0$, on average $h(D)$ is roughly $\sqrt{D}/\pi$ \cite{buell}, page 84.
A discriminant $D$ is \textit{fundamental} if either $D = 1 \bmod 4$ and $D$ square-free, or $D = 4m$ for some $m$ with $m = 2,3 \bmod 4$ and $m$ square-free. $D$ is called \textit{non-fundamental} (non-fun) otherwise. In most number theory papers it is assumed that $D$ is fundamental, but for our purposes we will mainly focus on non-fundamental discriminants instead.

%The next proposition is a collection of results about $h(D)$ in the case that $D$ is negative. REFS AND CHECK NEEDED (o.a. BUELL pag 84).

%\begin{prop}
%\label{class number bounds}
%    Let $D < 0$ be a discriminant, then:
%    \begin{enumerate}[label=$\alph*)$]
%        \item $h(D)$ is on average $\sqrt{D}/\pi$.
%        \item There exist $c,D_0 \in \N$ such that for all $D \geq D_0$ we have that $h(D) \leq c \sqrt{D} \ln(D)$.
%        \item If the extended Riemann hypothesis is true, then there exist $c,D_0 \in \N$ such that for all $D \geq D_0$ we have that $h(D) \leq c \sqrt{D} \ln\ln(D)$.
%    \end{enumerate}
%\end{prop}

\subsection{Non-fundamental discriminants}
\label{ss: non fun}
We will mainly follow Chapter 7 of Buell \cite{buell} for this subsection. Let $D$ be a discriminant and $r$ a positive integer. We will discuss how the class groups $C(D)$ and $C(Dr^2)$ are related.
For this we use \textit{transformation matrices} $A$, which are $2 \times 2$ integer matrices with $\det(A) = r$.

\begin{prop}
\label{embedding}
\hfill
\begin{enumerate}[label=$\alph*)$]
    \item Given any primitive form $f$ of discriminant $Dr^2$, there exists a form $g$ of discriminant $D$ and a transformation matrix $A$ with $\det(A) = r$ such that $f(x,y) = g(A \cdot (x,y)^T)$.
    \item If $f_1$ and $f_2$ are primitive equivalent forms of discriminant $Dr^2$, then there exists a primitive form $g$ of discriminant $D$, together with transformation matrices $A_1$, $A_2$ such that $A_1 A_2^{-1} \in \Gamma$ and:
    \[ \det(A_1) = \det(A_2) = r, \quad g(A_1 \cdot (x,y)^T) = f_1, \quad g(A_2 \cdot (x,y)^T) = f_2. \]
\end{enumerate}
\end{prop}

\begin{proof}
    See Proposition 7.1 in \cite{buell}.
\end{proof}

Proposition \ref{embedding} basically says that for every primitive form $f$ in the group of larger discriminant $Dr^2$, there is a unique (up to equivalence) primitive form $g$ in the group of smaller discriminant $D$ and a transformation matrix $A$ of determinant $r$ such that $f(x,y) = g(A \cdot (x,y)^T)$. In this case we say that $f$ is \textit{derived} from $g$.

Buell also makes those transformations explicit, for this he uses the following notion. Define two transformation matrices $A_1$ and $A_2$ of determinant $r$ to be \textit{right-equivalent} if there exists a matrix $B \in \Gamma$ such that $A_1B = A_2$.
It is easy to see that this is an equivalence relation. Furthermore, if $g$ is a form of discriminant $D$ and $f_1, f_2$ are forms derived from $g$ using right-equivalent transformations, then $f_1$ and $f_2$ are equivalent.

We first restrict ourselves to the case that $r$ is a prime $p$. The general case will be partially discussed in Appendix \ref{appendix: composite square factor}.

\begin{prop}
    The right-equivalent transformations of determinant $p$ have as equivalence class representatives the $p + 1$ transformations:
    \[ \begin{pmatrix}
        p & h\\
        0 & 1
    \end{pmatrix}
    \quad \text{for } \; 0 \leq h \leq p-1 \quad \text{and} \quad
    \begin{pmatrix}
        1 & 0\\
        0 & p
    \end{pmatrix}. \] 
\end{prop}

\begin{proof}
    See Proposition 7.2 in \cite{buell}.
\end{proof}

This means that given a form $g = (a,b,c) \in C(D)$ with $\gcd(a,p) = 1$, we can create the following $p+1$ forms of discriminant $Dp^2$:
\begin{equation}
\label{derived forms}
\begin{gathered}
    (ap^2, p(b + 2ah), ah^2 + bh + c) \quad \text{for } \; 0 \leq h \leq p-1\\
    \text{and } (a,bp,cp^2)
\end{gathered}
\end{equation}
and no others, up to equivalence.
In the last case, we say that $h = \infty$ was used in the transformation.
Note that these forms are not necessarily reduced, even when $g$ is reduced.

Using the formulas from $\eqref{derived forms}$, we can also define a map that goes the other way. Given a form $f \in C(Dp^2)$, find a form $f_2$ of the form $(ap^2,bp,c)$ that is equivalent to $f$. Then define $\pi(f) = (a,b,c) \in C(D)$. This map is well-defined (up to equivalence) because of Proposition \ref{embedding}.

The next question is whether the forms from \eqref{derived forms} are primitive and if they can be equivalent to each other.
We will use the \textit{Kronecker symbol} (\cite{cox} page 93), which we denote by $\legendre{D}{p}$, for given integers $D,p$.
Using this symbol, we can state the relation between $h(D)$ and $h(Dp^2)$. From this points onwards, we will focus on negative discriminants, because the important proposition below is not true for positive discriminants.

\begin{prop}
\label{prime square order}
    Given a form $(a, b, c)$ of discriminant $D < -4$ and an odd prime $p$ with $\gcd(a,p) = 1$, the $p + 1$ representative transformations of determinant $p$ produce exactly $p-\legendre{D}{p}$ primitive forms of discriminant $Dp^2$ which are all pairwise inequivalent.
    It follows that: 
    \[ h(Dp^2) = h(D) \cdot (p - \legendre{D}{p}) \]
\end{prop}

\begin{proof}
    See Propositions 7.3 and 7.4 in \cite{buell}.
\end{proof}

By applying this formula repeatedly, we get the following corollary:

\begin{cor}
\label{composite square order}
    Given a discriminant $D < -4$ and an odd integer $r = \prod_{i=1}^k p_k^{e_k}$.
    Define
    \[ \varphi_D(r) = \prod_{i=1}^k p_k^{e_k-1} (p_k - \legendre{D}{p_k}). \]
    Then we have the following relation:
    \[ h(Dr^2) = h(D) \cdot \varphi_D(r) \eqno \qed \]
\end{cor}

An important property of the embedding of $C(D)$ in $C(Dr^2)$ is that the derivedness property behaves well under composition of forms:

\begin{prop}
\label{derived behaves well}
    If $f_1,f_2 \in C(Dr^2)$ are derived from $g_1,g_2 \in C(D)$ respectively, then $f_1 \cdot f_2$ is derived from $g_1 \cdot g_2$.
\end{prop}

\begin{proof}
    See Proposition 7.9 in \cite{buell}. Buell's map $1_{\Delta}$ is the map $\pi$ in our context.
\end{proof}

\begin{cor}
\label{order derived form}
    Suppose $f \in C(Dr^2)$ is derived from $g \in C(D)$, where $r$ is odd. Let $a$ be the order of $g$, then the order of $f$ divides $a \cdot \varphi_D(r)$.
\end{cor}

\begin{proof}
    The forms in $C(Dr^2)$ derived from $e \in C(D)$ form a subgroup in $C(Dr^2)$.
    We know from Corollary \ref{composite square order} that the order of this subgroup is $\varphi_D(r)$.
    Proposition \ref{derived behaves well} implies that $f^a$ is derived from $e$.
    Therefore, $f^{a \cdot \varphi_D(r)} = e$.
\end{proof}

\section{Application to square-free decomposition}
\label{s: main algorithm}

\subsection{Forms derived from the identity element $e$}
\label{ss: forms derived from e}

From this point onwards, $D$ is a positive integer and $p$ is a prime.
All our discriminants will contain a factor 4, this is to make sure that $-4D$ is a discriminant.

In this section we will introduce an algorithm that computes the square-free decomposition of a given integer that uses class groups with non-fundamental discriminants. One of the main ingredients is the following proposition.
This proposition contains the very restricting condition that $D > p^2$. But, we will later see in Lemma \ref{r truc} that we can drop this assumption.

\begin{prop}
\label{derived from e}
    Suppose we have a reduced form $f \in C(-4Dp^2)$ that is derived from $e_{-4D} = (1,0,D) \in C(-4D)$ and $f \not \sim e_{-4Dp^2} = (1,0,Dp^2)$. Furthermore, suppose that $D > p^2$. Then
    \[ f = (p^2, 2pk, k^2+D) \]
    for some $-p/2 \leq k \leq p/2$.
\end{prop}

\begin{proof}
    Using the formulas from \eqref{derived forms}, we see that $f$ is equivalent to
    \[ g = (p^2, 2ph, h^2+D) \]
    for some $0 \leq h \leq p-1$. Let's check if $g$ is reduced. We have $h^2+D \geq D > p^2$, so that is a good start.\\
    If $h \leq p/2$ then $|2ph| \leq p^2$, so in this case $g$ is reduced. Reduced forms are unique, therefore $f = (p^2, 2ph, h^2+D)$ (not just equivalent, really equal).\\
    If $h \geq p/2$ then we do one reduction step with
    $ A = \left(\begin{smallmatrix}
    1 & -1\\0 & 1
  \end{smallmatrix}\right)$: 
    \[ g \sim g_2 = (p^2, 2ph - 2p^2, p^2 - 2ph + h^2+D) = (p^2, 2p(h - p), (h-p)^2+D). \]
    Then since $|2p(h - p)| \leq |2p(p/2)| = p^2$, we see that $g_2$ is reduced. Therefore $f = (p^2, 2p(h - p), (p-h)^2+D)$.
\end{proof}

The main takeaway from this proposition is the following factorization plan.
If we are able to find a non-trivial form in the larger group $C(-4Dp^2)$, that is derived from the trivial form in the smaller group $C(-4D)$, then we can find $p^2$ by reading off the first coefficient of the reduced form of $f$.
As mentioned before, the condition that $D > p^2$ will be dropped later.

A version of this proposition can also be stated for composite $p$. However, there are special cases where the complete factor $p$ can't be read off immediately from the first coordinate.
We show how to deal with this in Appendix \ref{appendix: composite square factor}. We recommend finishing this chapter first before reading that appendix.

%The main takeaway from this proposition (and Appendix \ref{appendix: composite square factor}) is that given $n=a^2b$, if we are able to find a non-trivial form in the larger group $C(-4n)$, that is derived from the trivial form in the smaller group $C(-4b)$, then we are able to produce the square-free decomposition of $n$ by reading off the first coefficient of the reduced form of $f$.
%As mentioned before, the condition that $b > a^2$ will be dropped later.

It is good to mention that this factorization idea in a general sense is very similar to other integer factorization algorithms. For example, in the elliptic curve method (ECM) by Lenstra \cite{lenstraECM}, a non-trivial point $P$ on an elliptic curve $E(\Z/n\Z)$ is constructed such that $P$ reduced mod $p$ is the identity element in the group $E(\Z/p\Z)$ of smaller size, where $p \mid n$. By computing $\gcd(P_x,n)$, the factor $p$ of $n$ will be found.
This will not be the last time that we compare our algorithm to the ECM, since there are many similarities.

The next step is to find a form $f$ that meets the requirements from Proposition \ref{derived from e}. For this we use one of the many algorithms of Lenstra.

\subsection{Schnorr and Lenstra class group factorization}
\label{ss: lenstra and schnorr}

In 1984, Schnorr and Lenstra published a general purpose integer factorization algorithm which they claimed could heuristically factor an integer $n$ in $\bigo(L_n[1/2,1])$ \cite{schnorrlenstra}. We will briefly discuss how the algorithm works.

Let $n$ be the number you want to factor, possibly square-free. Consider the class group $C(-4n)$ and take a random form $f \in C(-4n)$.
Construct a number $k$ that consists of powers of all primes up to some bound $B$. Now, if $f^k = e$ and the order of $f$ is even, then $g = f^{k/2^m}$ has order $2$ for some integer $m$.
Forms of order 2 are called \textit{ambiguous}.
They are of the form $(m,0,\frac{n}{m})$, with $\gcd(m,\frac{n}{m}) = 1$, which makes it easy to retrieve a factor of $n$ from them.
If this was successful and if $n$ is not yet completely factored, then we go again with a different $f$ and hopefully find another factor of $n$.

This method works as long as the order of $C(-4n)$ is smooth. If it is not smooth, then we can try again by considering the class group $C(-4ns)$ for some small positive integer $s$.
An ambiguous form in that group will still lead to the factorization of $n$.
Repeat this process for different values of $s$ until the factorization of $n$ is found.

%Below is `stage 1' of the algorithm we just described. It also has a `stage 2', which we will discuss in Appendix \ref{appendix: stage 2}.

\begin{algorithm}[H]  % zie http://tug.ctan.org/macros/latex/contrib/algorithmicx/algorithmicx.pdf
\caption{Schnorr and Lenstra stage 1}
\label{lenstra algorithm}
\begin{algorithmic}[1]
\Function{GeneralClassGroupFactorization}{$n$}
    \State $e = \sqrt{\ln(n)/\ln\ln(n))}$
    \State $B = n^{1/(2e)}$ \Comment{Prime bound}
    \State compute the first $t$ primes $p_1, \dots, p_t$ up to $B$
    \State $k = \prod_{i=2}^t p_i^{e_i}$, where $e_i = \max\{v : p_i^v \leq p_t^2\}$
    \State $s = 1$
    \While{no factorization found}
        \State pick random $f \in C(-4ns)$
        \State $g = f^k$
        \If{$g = e$ (the identity)}
            \State go back and pick a new $f$
        \EndIf
        \For{$1 \leq i \leq \log_2(\sqrt{n})$}
            \State $g = g^2$
            \If{$g = e$}
                \State construct ambiguous forms
                \State \Return{complete factorization of $n$}
            \EndIf
        \EndFor
        \State update $s$ to be the next square-free number after $s$
    \EndWhile
\EndFunction
\end{algorithmic}
\end{algorithm}

Schnorr and Lenstra made the following heuristic assumptions. Assumptions $a,b$ are $(2.1)$, $(2.2)$ in \cite{schnorrlenstra} respectively. Assumption $c$ is mentioned later in their article in equation $(4.3)$.
This final assumption turned out to be wrong, which we will see later in this chapter. Let $\Psi(x,y)$ be the number of positive integers $a \leq x$ such that $a$ is $y$-smooth.

\begin{assumptions}
\label{lenstra assumptions}
\hfill
    \begin{enumerate}[label=$\alph*)$]
    \item The order of a class group of discriminant $D$ is at least as likely to be smooth as a random integer of size $\sqrt{D}$. More precisely, for all $n,t$:
    \[ \# \{m \leq n : h(-m) \mid \prod_{i=1}^t p_i^{e_i} \} / (0.5n)
    \geq \# \{m \leq \sqrt{n} : m \mid \prod_{i=1}^t p_i^{e_i} \} / \sqrt{n}, \]
    where the $e_i$ are defined in Algorithm \ref{lenstra algorithm}.
    \item A significant portion of integers are smooth.
    More explicitly, for all $n$ and $e \leq \sqrt{\ln(n)/\ln\ln(n))}$ we have that $\Psi(n,n^{1/e})/n \geq e^{-e}$.
    \item Given an integer $n$, the smoothness bounds of $h(-4ns)$ are independent for all square-free integers $s$. 
    \end{enumerate}
\end{assumptions}

It is good to mention that the Cohen-Lenstra heuristics \cite{cohenlenstraheuristics} suggest that the odds of finding a class group with a smooth order is actually a bit better than that of a random integer of the same size.

We can now state the main result from Schnorr and Lenstra.

\begin{theorem}
\label{lenstra theorem}
    Assume Assumptions \ref{lenstra assumptions}. Then for all composite integers $n$, Algorithm \ref{lenstra algorithm} will completely factor $n$ in expected $\bigo(L_n[1/2,1])$ time.
\end{theorem}

\begin{proof}
    See Theorem 5 and the run time analysis section of Chapter 4 from \cite{schnorrlenstra}.
\end{proof}

Unfortunately, in 1992 it was found by Lenstra and Pomerance \cite{lenstra_flaw} (Chapter 11), that the above stated run time was incorrect for a large set of numbers. Can you guess which? Numbers that have a large prime square divisor!

Suppose that $n$ has a divisor $p^2$, where $p$ is prime. Suppose furthermore that $p-1$ and $p+1$ are both not smooth. Then by Proposition \ref{prime square order}, we see that $C(-4ns)$ is divisible by either $p-1$ or $p+1$ for all integers $s$. Therefore, there will be no $s$ such that $C(-4ns)$ is smooth.

Interestingly enough, we will show that by adapting Lenstra's algorithm for integers of this form, we actually get an algorithm that is \emph{faster} than the originally claimed time bound in this special case.

We have now know that assumption $c$ is incorrect. Therefore, we will use the following assumptions instead, where we only tweak assumption $c$.

\begin{assumptions}
\label{updated assumptions}
\hfill
    \begin{enumerate}[label=$\alph*)$]
    \item The order of a class group of discriminant $D$ is at least as likely to be smooth as a random integer of size $\sqrt{D}$. More precisely, for all $n,t$:
    \[ \# \{m \leq n : h(-m) \mid \prod_{i=1}^t p_i^{e_i} \} / (0.5n)
    \geq \# \{m \leq \sqrt{n} : m \mid \prod_{i=1}^t p_i^{e_i} \} / \sqrt{n}, \]
    where the $e_i$ are defined in Algorithm \ref{lenstra algorithm}.
    \item A significant portion of integers are smooth.
    More precisely, for all $n$ and $e \leq \sqrt{\ln(n)/\ln\ln(n))}$ we have that $\Psi(n,n^{1/e})/n \geq e^{-e}$.
    \item Given a square-free integer $n$, the smoothness bounds of $h(-4ns)$ are independent for all square-free integers $s$. 
    \end{enumerate}
\end{assumptions}

These new assumptions are still just conjectures, but now at least the problem with large square divisors is now taken care of. In Section \ref{ss: speed in practice}, we will give experimental results that might convince one of the correctness of these assumptions.

We will rephrase Theorem \ref{lenstra theorem} so that it uses the new assumptions and it can be used for the analysis of our algorithm.

\begin{theorem}
\label{lenstra refrase}
    Assume Assumptions \ref{updated assumptions}. Let $n$ be square-free and composite.
    %Let $B$ be the prime bound.
    %Per multiplier $s$, Algorithm \ref{lenstra algorithm} takes $\bigo(B)$ compositions.
    %Once a suitable $s$ is found, another $\bigo(\ln(s) (B + s))$ compositions are needed to construct the ambiguous forms.
    % Misschien niet helemaal waar. In Lenstra's artikel gebruikt hij opeens p_t = B*sqrt(r), ipv p_t = B
    Let 
    \[ B = n^{1/(2e)} \in \bigo(L_n[1/2,1/2]) \]
    as stated in Algorithm \ref{lenstra algorithm}.
    Then per multiplier $s$, the algorithm performs $\bigo(B)$ compositions.
    It takes expected $\bigo(B)$ tries to find a suitable $s$.
    Afterwards, another $\bigo(\ln(s) (B + s))$ compositions are needed to construct the ambiguous forms.
    In total, it takes expected $\bigo(B^2) = \bigo(L_n[1/2,1])$ compositions to factor $n$.
    %For the prime bound $B$ stated in the algorithm, it takes expected $\bigo(B^2) = \bigo(L_n[1/2,1])$ compositions to find a suitable $s$ and another $\bigo(B^2) = \bigo(L_n[1/2,1])$ compositions to construct the ambiguous forms.
    %Assume Assumptions \ref{updated assumptions}. Then for all composite square-free integers $n$, Algorithm \ref{lenstra algorithm} will find integers $s,k$ and a form $f \in C(-4ns)$ such that $f^k = e = (1,0,ns)$,  in expected $\bigo(L_n[1/2,1])$ compositions.
    %Afterwards, another $\bigo(L_n[1/2,1])$ compositions are needed to construct the ambiguous forms, which give the complete factorization of $n$
\end{theorem}

\begin{proof}
    We again refer to the time analysis section of Chapter 4 of \cite{schnorrlenstra}.
\end{proof}

\subsection{A new square-free decomposition algorithm}

Using the results from Sections \ref{ss: forms derived from e} and \ref{ss: lenstra and schnorr}, we can now state the first version of our new algorithm that finds the square-free decomposition of a given integer.

From this point onwards, $n = a^2b$ is an integer with $b$ square-free.
The main idea is to use Algorithm \ref{lenstra algorithm} to find a form $f \in C(-4n)$ that is derived from $e_{-4b} \in C(-4b)$.
Then use Proposition \ref{derived from e}, or Appendix \ref{appendix: composite square factor} if $a$ is composite, to find the factor $a$ of $n$. From this the square-free decomposition can be easily computed.

As mentioned before, the assumption that $b > a^2$ is quite limiting. Because then $h(-4b)$ will be fairly large, making the algorithm not very efficient.
Fortunately, there is a great way to circumvent this problem, by introducing an integer $r$ with known factorization. %Consider the larger and even less fundamental group $C(-4nr^2)$.
In the next lemma we show that we can enlarge $b$ with a factor $r^2$ with minimal drawbacks.
This will help a lot, since then the form $f$ in Proposition \ref{derived from e} will be reduced, even if the original $b$ is not large compared to $a$.

\begin{lemma}
\label{r truc}
    Suppose $g \in C(-4n)$ is derived from $e \in C(-4b)$.
    Let $r$ be a positive integer with $\gcd(r,a) = 1$.
    Lift $g$ to some $h \in C(-4nr^2)$ using the formulas in \eqref{derived forms}.
    Then $l = h^{\varphi_{-4n}(r)}$ is not only derived from $e \in C(-4b)$, but also from $e \in C(-4br^2)$.
\end{lemma}

\begin{proof}
    First note that $h$ is derived from $e \in C(-4b)$.
    Suppose $h \in C(-4nr^2)$ is derived from $h_2 \in C(-4br^2)$.
    Then $h_2$ is also derived from $e \in C(-4b)$, because Proposition \ref{embedding} implies that $h$ is derived from a unique form in $C(-4b)$.
    Corollary \ref{order derived form} now implies that the order of $h_2$ divides $\varphi_{-4b}(r)$.

    \[\begin{tikzcd}
    	{C(-4b)} && {C(-4n)} & {} \\[-4ex]
    	e && g \\
    	\\[+3ex]
    	{h_2} && h \\[-4ex]
    	{h_2^{\varphi_{-4n}(r)}} && {l = h^{\varphi_{-4n}(r)}} \\[-4ex]
    	{C(-4br^2)} && {C(-4nr^2)}
    	\arrow[maps to, from=2-1, to=4-1]
    	\arrow[maps to, from=2-1, to=2-3]
    	\arrow[maps to, from=2-3, to=4-3]
    	\arrow[maps to, from=4-1, to=4-3]
    	\arrow[maps to, from=5-1, to=5-3]
    \end{tikzcd}\]
    
    \noindent Write $r = \prod_{i=1}^k p_k^{e_k}$, then
    \[ \varphi_{-4b}(r) = \prod_{i=1}^k p_k^{e_k-1} (p_k - \legendre{b}{p_k}) = \prod_{i=1}^k p_k^{e_k-1} (p_k - \legendre{n}{p_k}) = \varphi_{-4n}(r) \]
    since $\gcd(r,a) = 1$.
    Using Proposition \ref{derived behaves well}, we see that $l$ is derived from
    \[ h_2^{\varphi_{-4n}(r)} = h_2^{\varphi_{-4b}(r)} = e \in C(-4br^2). \qedhere\]
\end{proof}

In the algorithm below, $n = a^2b$ is an integer not divisible by $3$ and $b_2$ is an upper bound for $b$. If no good upper bound is known, then a small value can be used initially, and it can be incrementally increased if no factorization is found.
The first part of the algorithm is completely the same as Algorithm \ref{lenstra algorithm}.
The difference is that if Lenstra's algorithm is unable to find the factorization, then we do some additional steps which might lead to the square-free decomposition of $n$. 

\begin{algorithm}[H]  % zie http://tug.ctan.org/macros/latex/contrib/algorithmicx/algorithmicx.pdf
\caption{Square-free decomposition stage 1}
\label{decomposition algorithm 1}
\begin{algorithmic}[1]
\Function{SquareFreeDecomposition}{$n,b_2$}
    \State $e = \sqrt{\ln(b_2)/\ln\ln(b_2))}$
    \State $B = b_2^{1/(2e)}$ \Comment{Prime bound}
    \State compute the first $t$ primes $p_1, \dots, p_t$ up to $B$
    \State $k = \prod_{i=2}^t p_i^{e_i}$, where $e_i = \max\{v : p_i^v \leq p_t^2\}$
    \State $s = 1$
    \State $r = 3^{\lceil \log_3(\sqrt{n})) \rceil}$
    \While{no factorization found}
        \State pick random $f \in C(-4ns)$
        \State $g = f^k$
        \If{$g = e$ (the identity)}
            \State go back and pick a new $f$
        \EndIf
        \For{$1 \leq i \leq \log_2(\sqrt{n})$}
            \State $g = g^2$
            \If{$g = e$}
                \State construct ambiguous forms
                \State \Return{complete factorization of $n$}
            \EndIf
        \EndFor
        \State lift $g$ to a form $h \in C(-4nsr^2)$ using the formulas in \eqref{derived forms}
        \State $l = h^{\varphi_{-4ns}(r)}$
        \State try to find $a^2$ using the form $l$ and Proposition \ref{derived from e} or Appendix \ref{appendix: composite square factor}
        \If{$a^2$ is found}
            \State $a = \sqrt{a^2}$ and $b = n/a^2$
            \State \Return{$a,b$}
        \EndIf
        \State update $s$ to be the next square-free number after $s$
    \EndWhile
\EndFunction
\end{algorithmic}
\end{algorithm}

\begin{theorem}
\label{decompostion theorem 1}
    Assume Assumptions \ref{updated assumptions}. Then for all integers $n = a^2b$, Algorithm \ref{decomposition algorithm 1} will find the square-free decomposition (and possibly the full factorization) of $n$ in expected time:
    %\[ \tbigo{(L_{b}[1/2,1] \ln(n) + \ln(n)^2)} = \tbigo{(e^{(1 + o(1)) \sqrt{\ln(b) \ln\ln(b)}} \ln(n) + \ln(n)^2)}. \]
    \[ \tbigo{(L_{b}[1/2,1] \ln(n) + L_{b}[1/2,1/2] \ln(n)^2)} \]
\end{theorem}

\begin{proof}
    Let's first look at the correctness of the algorithm.
    We know from Algorithm \ref{lenstra algorithm} that the code up to line 20 has a chance to completely factor the integer $n$. This is the case if $g$ is the trivial form in the group $C(-4ns)$.
    If that is not the case, then using Proposition \ref{derived behaves well}, we see that there is still a possibility that $g$ is derived from the form $e$ in the underlying group $C(-4bs$).\\
    In that case, we know from Lemma \ref{r truc} that the form $l$ in the larger group $C(-4nsr^2)$ is not just derived from $e \in C(-4bs)$, but also from $e \in C(-4bsr^2)$.
    Thus, we can now apply Proposition \ref{derived from e} or Appendix \ref{appendix: composite square factor} with the form $l$ to find $a^2$, because 
    \[ bsr^2 \geq r^2 \geq n \geq a^2. \]
    
    Now let's look at the running time of the algorithm.
    %Picking a random form $f$ can be done in expected $\tbigo(\ln(n))$ by checking for the first few primes $q$ if $-4s$ is square mod $q$.
    %Apart from line 7, the code until line 21 is the same as Algorithm \ref{lenstra algorithm}. 
    Until line 20, the code is the same as Algorithm \ref{lenstra algorithm}, except that our prime bound depends on $b_2$ instead of $n$.
    If $b_2$ is a good approximation of $b$, then Proposition \ref{derived behaves well} and Theorem \ref{lenstra refrase} imply that we can expect to get a form derived from $e$ in $\bigo(L_b[1/2,1])$ group operations, i.e. $\tbigo{(L_b[1/2,1] \ln(n))}$ steps if we use a fast composition algorithm \cite{schonhagefastcomposition}.
    Finally, line 21 and 22 can be done in $\tbigo(\ln(n)^2)$ per value of $s$ that we try, since the discriminant of $h$ is $\bigo(-n^2s)$ and the exponent is $\bigo(\sqrt{n})$.
    From Theorem \ref{lenstra refrase} we also know that the expected number of multipliers $s$ that we have to try is $\bigo(L_{b}[1/2,1/2])$.
    Hence, the total expected running time of the algorithm in this case is 
    \[ \tbigo{(\psi_b(n))} := \tbigo{(L_{b}[1/2,1] \ln(n) + L_{b}[1/2,1/2] \ln(n)^2)}. \]
    If a good approximation of $b$ is not known, then start with a small value for $b_2$ and try to find the factor $a^2$ in $\tbigo{(\psi_{b_2}(n))}$ time.
    If the algorithm is not successful in that time frame, then try again with $b_2 = 2b_2$, etc. A $b_2$ of size $b$ is found within $\log_2(b)$ steps.
    Thanks to the definition of $L_b$, the total complexity is still $\tbigo{(\psi_b(n))}$.
\end{proof}

% Opm: De bound $\tbigo{(L_{b}[1/2,1] + \ln(n)^2)}$ is niet haalbaar. Want b kan bijv O(ln(n)^4) zijn, dan is e^{\sqrt{\ln(b) \ln\ln(b)}} \in O(e^{2 ln(ln(n))}) = O(ln(n)^2). Dan maakt de factor log(n) nog wel uit

\begin{cor}
\label{clean runtime}
    Fix $\epsilon > 0$ and assume Assumptions \ref{updated assumptions}. Then for all integers $n = a^2b$ with $b = n^{1/\epsilon}$, Algorithm \ref{decomposition algorithm 1} will find the square-free decomposition (and possibly the full factorization) of $n$ in expected time:
    \[ \bigo{(L_b[1/2,1])} = \bigo{(e^{(1 + o(1)) \sqrt{\ln(b) \ln\ln(b)}})} = \bigo{(L_n[1/2,\sqrt{1/\epsilon}])}. \]
\end{cor}

\begin{proof}
    First note that for any constant $c$
    \[ L_b[1/2,c] = L_n[1/2, c \sqrt{1/\epsilon}]. \]
    Therefore, we can hide the factors $\ln(n)$ of $\psi_b(n)$ in the $o(1)$ terms to get
    \[ \tbigo{(\psi_b(n))} = \tbigo{(L_{b}[1/2,1] + L_{b}[1/2,1/2])} = \tbigo{(L_{b}[1/2,1])}. \]
    Finally, we can drop the $\sim$ in $\tbigo$ for the same reason.
\end{proof}

\begin{remark}
\normalfont
    Let's make some remarks about our new factorization algorithm.
    \begin{itemize}
        %\item If we have to do the incremental approach with $b_2$, then after updating $b_2$, it is better to continue with the same values of $s$ that we tried before instead of increasing $s$. This is because $h(-4ns) \approx \sqrt{ns}$, so you don't want to use large values of $s$.

        \item In cryptographic applications, integers $n = p^2q$ are sometimes used, where $p,q$ are primes of roughly the same size. Thus a good approximation of $b \approx n^{1/3}$ is then known.
        In cases like these we can also choose $r$ smaller: $r = cn^{1/6}$, for some small $c$ that depends on how much $p$ and $q$ differ.
        Using this $r$, we still meet the requirement of Proposition \ref{derived from e}.
        More generally, numbers of the $n = p^kq$ are sometimes used in cryptographic systems. Our algorithm is especially fast when $k$ is even, because then the square-free part of $n$ is only $q$, compared to $pq$ when $k$ is odd.

        \item The order in which we try the different multipliers $s$ might not be optimal in the current presentation.
        The optimization of this seems to be somewhat of an art, we analyze this (experimentally) in Appendix \ref{appendix: multipliers}.
        
        %\item The additional factor $r$ that we introduced can be taken somewhat smaller if a good upper bound of $b$ is known, say $b_2 \leq 2b$. Then we can take $r = 2\sqrt{n}/b_2$, since we then still have that $bsr^2 \geq 4bn/(b_2)^2 \geq 4a^2/4 = a^2$.
        
        \item Instead of taking $r = 3^m$, we could also take $r = 2^m$ for a suitable $m$ (we did not state Corollary \ref{composite square order} for this case), or $r$ could be $r = \text{NextPrime}(\sqrt{n})$.
        The good thing about taking $r$ prime is that $r$ will only a little bit bigger than $\sqrt{n}$. With $r = 3^m$ for example, you usually overshoot it by quite a bit.
        Smaller $r$ speed up the arithmetic in the group $C(-4nsr^2)$, which is not very important for this version of the algorithm, but it will make a practical difference in Chapter \ref{s: stage 2}.
        However, the complexity analysis of the NextPrime function is somewhat messy, so we left it out of this presentation.
        
        \item If we reach line 26 and find the factors $a$ and $b$, then we can also find the complete factorization of $b$.
        This is because we can run Algorithm \ref{lenstra algorithm} with input $b$, using the same $k$ and $s$ that were used to find $a$ and $b$.
        This does not produce the full factorization of $n$, because we haven't factored $a$ yet. But, $a \leq \sqrt{n}$, so in most cases it will be fairly easy to factor it.
        
        \item As mentioned before, Lenstra's original algorithm struggled with numbers $n$ that have large prime square divisors. We now see that in that case we actually have a faster way to find at least the square-free decomposition of $n$.
        We can also restore Lenstra's algorithm to its former glory by amending it as follows.
        Given $n$, possibly square-free, run Algorithm \ref{decomposition algorithm 1}.
        If an ambiguous form is found, then compute the full factorization of $n$ and stop.
        If factors $a,b$ are found such that $n = a^2b$, then completely factor $b$ using Lenstra's original algorithm, which is possible since $b$ is square-free.
        It is possible that $a$ does contain a square factor. Therefore, we repeat this algorithm recursively with the input $a$. There are at most $\log_2\log_2(n)$ recursion steps.
        This produces the full factorization of $n$ in expected 
        \[ \bigo{(L_{n}[1/2,1] \ln\ln(n))} = \bigo{(L_{n}[1/2,1])} \]
        which is the original running time that Schnorr and Lenstra claimed!
        
        \item After we discovered the trick of introducing the factor $r$, we later found out that \cite{truc_met_r} had already done in 2009.
        In Theorem 3 of that article, Castagnos and Laguillaumie show how to find the square part of the discriminant given a form that is derived from $e$ in the underlying group.
        Another article from 2009 \cite{Factoring_pq2} does something similar. Using Coppersmith's method, they show that when given a form that is derived from a form of small norm, you can find the square part of the discriminant. An example of such a form of small norm is of course $e$.\\
        A natural question for both articles is therefore: how can we find a form that is derived from $e$?
        As we now have seen, Lenstra's algorithm answers this question quite well, which raises the question why this combination has not been spotted before. One possible explanation is that Schnorr and Lenstra's algorithm is from 1984, whereas these articles are much more recent.
        
        %\item As mentioned in the bullet above, it is also possible to find the square-free decomposition of an integer by using Coppersmith's method instead of lifting a form and then raising it to a specific power like we do in line 22 of Algorithm \ref{decomposition algorithm 1}.
       % But this is not faster, since one call to the LLL algorithm is much slower. However, it might not even impact the running time of the overall algorithm that much, since computing an LLL reduced basis takes polynomial time, but computing $g = f^k$ in line 10 does not.
    \end{itemize}
\end{remark}

Now that we have a solid foundation for our factoring algorithm, we will look at similar algorithms and try to learn from them to optimize ours.

\section{Stage 2}
\label{s: stage 2}

\subsection{Introduction}

In Algorithm \ref{decomposition algorithm 1} we constructed a large $k$ and computed $g = f^k$ for some $f \in C(-4ns)$, we then hoped that this form is derived from $e$. If not, then we try the next $s$.
Our algorithm is an example of \textit{algebraic-group factorization}, which is a more general factorization strategy that works in some other groups as well, such as finite fields and elliptic curves.

In most other groups where this principle can be applied, there is something called a `stage 2'. If in stage 1 we used primes up to a bound $B$, then in stage 2 we extend our prime bound to $B_2 > B$. But, instead of doing exactly the same steps with a larger $k$, we now continue with the $g = f^k$ that we computed, and we check for every prime $p \in [B,B_2]$ separately if $g^p$ is derived from $e$.
The idea behind this is that if stage 1 was unsuccessful, then there is a good chance that we are just missing one prime factor $q$ of the order of $C(-4bs)$. Then since we only are examining one prime $p$ at the time, we might be able to take $B_2$ quite a bit larger than $B$, without having to increase the runtime by a lot. If the factor $q$ lies in the interval $[B,B_2]$, then we find the square-free decomposition of $n$.

Note that if there are two (or more) primes $q_1,q_2 \in [B,B_2]$ that we are missing from the order of $C(-4bs)$, then we are not going to factor $n$ using stage 2, since we only look at one prime at the time. In stage 1 this problem does not exist. Nevertheless, we will see that on average, having a stage 2 improves the performance of the algorithm.

We of course want to take $B_2$ as large as possible, without having it impact the runtime too much. In Section \ref{ss: generic stage 2}, we look at a method where we can take $B_2$ to be roughly $B \ln(B)$.
In Appendix \ref{appendix: stage 2}, we will look at algebraic-group factorization algorithms that are able to use $B_2 = B^2$. Unfortunately, we have not found such a method for our algorithm, but in that appendix we do present possible ways that could lead to such a method.

\subsection{Generic stage 2}
\label{ss: generic stage 2}
The first known `stage 2' for algebraic-group factorization algorithms was already discovered in 1974 in Pollard's classic paper \cite{pollardstage2}, Chapter 4.
There, it was applied in the group $(\Z/n\Z)^*$, which uses the Chinese remainder theorem to see that this group has a subgroup $\F_p^*$, where $p$ is prime and $p \mid n$.
If an element in $b \in (\Z/n\Z)^*$ is found that is derived from $1 \in \F_p^*$, then the factor $p$ can be found. Derived in this context means that $b = 1 \bmod p$, hence $\gcd(b-1,n)$ will provide $p$.
Such $b$ can be found by computing $b = a^k \bmod n$ for a large highly composite $k$ and some random $a$, precisely as what we saw in Algorithms \ref{lenstra algorithm}, \ref{decomposition algorithm 1}. Because of the order of the group $\F_p^*$, methods like Pollard's are called $p-1$ algorithms.

After computing $b = a^k \bmod n$, Pollard continues by first precomputing $b^2, b^4, \dots, b^{2m}$, where $2m$ is the largest gap between two consecutive primes in the interval $[B,B_2]$.
The best known unconditional bound on $m$ is $\bigo(B_2^{0.525})$ \cite{bakerprimegaps}.
However, most conjectured bounds are much smaller, for example Cramér's bound is $\bigo(\ln(B_2)^2)$ \cite{cramer}.

Let $p_{t+1} \leq \dots \leq p_u$ be the primes in the interval $[B,B_2]$.
We can now compute $b^r \bmod n$ for each prime $r \in [B,B_2]$ by first computing $c = b^{p_1} \bmod n$ in $\bigo(\ln(B))$ multiplications mod $n$.
Afterwards, we only have to do one multiplication mod $n$ to compute $b^{p_2} = c \cdot b^{p_2-p_1} \bmod n$, since we precomputed $b^{p_2-p_1} \bmod n$.
Continuing like this covers all primes in the interval $[B,B_2]$ in $\bigo(\pi(B_2) - \pi(B))$ multiplications mod $n$, where $\pi$ is the prime counting function.

The prime number theorem states that $\pi(x) \approx x/\ln(x)$ for 
all large positive integers $x$. Therefore, this stage 2 takes $\bigo(B_2/\ln(B_2))$ multiplications mod $n$.
If we take $B_2 = B \ln(B)$ then the number of multiplications becomes $\bigo(B)$.
This is the same number of multiplications as in stage 1.
Brent \cite{brent_stage_2} (Section 9.2) mentions that using this stage 2 improves the runtime of the $p-1$ algorithm by a factor roughly $\ln\ln(p)$.

%In Algorithm \ref{decomposition algorithm 1} we only have a stage 1, and the number of compositions that are needed to compute $f^k$ (including factors 2) is roughly $B$.
%This is because $k = \prod_{i=1}^t p_i^{e_i}$, with $p_i^{e_i} \approx B^2$. Hence $k \approx B^{2t}$, where $t = \pi(B)$, thus the number of compositions in stage 1 is:
%\[ \log_2(k) \approx 2\pi(B)\log_2(B) \approx 2B \in \bigo(B). \]
%We can therefore take $B_2$ to be a factor $\ln(B)$ bigger than $B$ without increasing the complexity of the overall algorithm.

We call this the `generic' stage 2 because it can be used in all groups that follow the algebraic-group factorization strategy. So, let's apply it to our algorithm.
We won't state Algorithm \ref{decomposition algorithm 1} here again, Algorithm \ref{algorithm generic stage 2} will take place between lines 27 and 28 of Algorithm \ref{decomposition algorithm 1}.

Unfortunately, the compositions of stage 2 will take place in the larger group $C(-4nsr^2)$, because otherwise we would have to lift the form $g^p$ to this group for every prime $p$ in the interval $[B,B_2]$, which is slower.
This is where it is useful to take $r$ as small as possible.

\begin{algorithm}[H]
\caption{Generic stage 2 continuation}
\label{algorithm generic stage 2}
\begin{algorithmic}[1]
\Function{Stage2}{$n,b_2$}
    \State $B_2 = B\ln(B)$
    \State compute the primes $p_{t+1}, \dots, p_u$ in the interval $[B,B_2]$
    \State $m = \max\limits_{t+1 \leq i \leq u-1}(p_{i+1}-p_i)$
    \State precompute $l^2, l^4, \dots, l^{m} \in C(-4nsr^2)$
    \State $p_{-1} = 0$  \Comment{the previous prime}
    \For{$p \in [p_{t+1}, \dots, p_u]$}
        \State $step = p-p_{-1}$
        \State $l_2 = l_2 \cdot l^{step}$  \Comment{if $p_{-1} \neq 0$ then $l^{step}$ is precomputed}
        \State try to find $a^2$ using the form $l_2$ and Proposition \ref{derived from e} or Appendix \ref{appendix: composite square factor}
        \If{$a^2$ is found}
            \State compute $a = \sqrt{a^2}$ and $b = n/a^2$
            \State \Return{$a,b$}
        \EndIf
        \State $p_{-1} = p$  \Comment{update the previous prime}
    \EndFor
\EndFunction
\end{algorithmic}
\end{algorithm}

In the next proposition we summarize what we know about Algorithm \ref{algorithm generic stage 2}.

\begin{prop}
    Assume that the largest prime gap in $[B,B_2]$ is polynomial in $\ln(B_2)$.
    Then the generic stage 2 continuation of Algorithm \ref{decomposition algorithm 1} does not change the complexity per $s$ that we try.
    The number of class groups that have to be tried to factor $n = a^2b$ this way is reduced by a factor of about $\ln\ln(b)$.
\end{prop}
% Note that $m$ depends on b, and not on n (which is a good thing)

The generic stage 2 continuation uses $B_2 = B\ln(B)$. This is not bad, but there are methods used in other algebraic-group factorization algorithms that use a whopping $B_2 = B^2$, or something close to it, and still only take $\bigo(B)$ group operations.
This provides a significant speedup, as can be seen in Proposition \ref{lenstra stage 2 speed-up}, so we want to be able to do this as well.
Unfortunately, we were not able to find such a good stage 2 for our algorithm. However, we do explore some interesting possible approaches in Appendix \ref{appendix: stage 2}.

\section{Timings and comparison to other algorithms}

\subsection{Complexity comparisons}
\label{ss: complexity comparisons}

In this section we will compare our algorithm to three other algorithms: the elliptic curve method (ECM), the number field sieve (NFS) and the lattice methods mentioned in Chapter 1.
Let's start with the ECM.

As we saw in Theorem \ref{decompostion theorem 1}, given an integer $n=a^2b$, we can compute the square-free decomposition of $n$ in expected time
\[ \tbigo{(L_{b}[1/2,1] \ln(n) + L_{b}[1/2,1/2] \ln(n)^2)} \]
The ECM works a bit differently, if $p$ is the smallest prime factor of $n$, then the ECM can find $p$ in expected $\tbigo(L_p[1/2,\sqrt{2}] \ln(n))$ time \cite{lenstraECM} if fast arithmetic is used.
Because of the $o(1)$ in the exponent of the $L_p$ function, the logarithmic improvement provided by the better stage 2 of the ECM is not shown in the complexity.

Both algorithms hope to find a group of smooth order. In our algorithm, we work with the class groups $C(-4bs)$, which have size roughly $\sqrt{b}$.
In the ECM, you work with elliptic curves $E(\F_p)$, which have size roughly $p$.
This is why the $\sqrt{2}$ term is not present in the $L$ function of the complexity of our algorithm.

It is good to mention that just like the ECM, there is a natural way to parallelize our algorithm.
If you have $m$ processors, give each of them a different $s$ and try to factor $n$ using the class groups $C(-4ns)$.
This will provide a linear speed-up in $m$.

If the factors $a,b$ of $n$ are not prime, then the ECM will most likely find \textit{some} factor of $n$ before our algorithm computes the square-free decomposition. However, the ECM might take longer to find the complete square-free decomposition, depending on the structure of the remaining prime factors of $n$.
If $n$ is of the form $n = p^2q$, where $p,q$ are primes, then purely looking at the asymptotic complexities, our method will be faster than the ECM when $p^2 > q$. In practice, we might need $p^2$ to be even larger compared to $q$, since the ECM has a better stage 2 and many more optimizations.

From now on, we will assume that $n$ is of the form $n = p^2q$, where $p,q$ are primes of roughly the same size.
We do this, because this is the hardest and probably the most common setting where our algorithm can be used.
As mentioned in Chapter 1, there are quite a few cryptographic systems that use numbers of this form, where they rely on the assumption that these numbers are hard to factor.
In this case $p^2$ is much larger than $q$, so our algorithm will be faster than the ECM when $n$ is large enough.

The number field sieve completely factors its input $n$ in expected time
\[ \bigo{(L_n[1/3,(64/9)^{1/3}])} = \bigo{(e^{(64/9 + o(1))^{1/3}\ln(n)^{1/3} \ln\ln(n)^{2/3}}}) \] \cite{crandallprimenumbers}, page 288.
From Corollary \ref{clean runtime}, we know that our method takes expected time 
\[ \bigo{(L_n[1/2,\sqrt{1/3}])} = \bigo{(e^{(\sqrt{1/3} + o(1)) \sqrt{\ln(n) \ln\ln(n)}}}) \]
The exponent of $\ln(n)$ in the complexity of the NFS is smaller than in ours. Therefore, when $n$ is large enough, the NFS will factor $n$ faster than our algorithm, even when $n$ is in this special form.
However, since the constant $(64/9)^{1/3} \approx 1.92$ is much bigger than $\sqrt{1/3} \approx 0.58$, our method will likely be faster up to some point.
If we solve for $n$ in the equation 
\[ L_n[1/3,(64/9)^{1/3}] = L_n[1/2,\sqrt{1/3}] \]
then we find the massive $n \approx 10^{5613}$.
This should be taken with a grain of salt, since this does not take into account any optimizations nor the $o(1)$ terms.
But, it seems reasonable to assume that for integers $n = p^2q$ of cryptographic size, our method will be faster than the NFS.
Table \ref{timing table} also supports this claim, because factoring integers of $150$ digits on a standard computer will take more than a few hours using the NFS.

Finally, let's have a look at the lattice methods that can factor integers of the form $n = p^rq$.
If $p$ and $q$ are primes of roughly the same size, then a lattice attack can factor $n$ in deterministic
$\tbigo(n^{1/(r+1)^2})$ time \cite{harvey_factoring_p^rq}, Section 1.3.
This means that for $r = 2$, this method takes $\tbigo(n^{1/9})$ time, which is not sub-exponential in $\ln(n)$ and therefore slower than our method if $n$ is large enough.
But, if $r$ is large enough, then the lattice methods will be faster.
Especially if $r$ is odd, since then the square-free part of $n$ is $pq$ instead of $p$, significantly slowing down our algorithm.

\subsection{Speed in practice}
\label{ss: speed in practice}

We implemented our algorithm in \textsc{Magma} and ran it on a fairly standard desktop ($\SI{3.35}{\giga \hertz}$).
Our basic implementation can be found in Appendix \ref{appendix: implementation}.
In Table \ref{comparison table} we compare three algorithms:
\begin{enumerate}
\item Our complete algorithm, i.e. the combination of Algorithm \ref{decomposition algorithm 1} with the generic stage 2 continuation presented in Algorithm \ref{algorithm generic stage 2}.
\item Only stage 1 of our algorithm, i.e. Algorithm \ref{decomposition algorithm 1}.
\item The elliptic curve method (ECM). We used the optimized built-in version of \textsc{Magma}. 
\end{enumerate}
The numbers that are were factored to create this table were of the form $n = p^2q$, where $p,q$ are primes of roughly the same size.

\begingroup
\renewcommand{\arraystretch}{2.5}
\begin{table}[H]
\centering
\begin{tabular}{l|ccccc}
 & $q \approx 10^{15}$ & $q \approx 10^{20}$ & $q \approx 10^{25}$ & $q \approx 10^{30}$ & $q \approx 10^{35}$\\ \hline
\makecell[l]{Mean time \\ with stage 2}      & 0.11 & 0.80 & 5.16 & 30.30 & 135.91\\
\makecell[l]{Median time \\ with stage 2}    & 0.07 & 0.50 & 3.54 & 23.04 & 80.52\\
\makecell[l]{Mean time \\ only stage 1}   & 0.16 & 1.63 & 11.65 & 66.71 & 357.74\\
\makecell[l]{Median time \\ only stage 1} & 0.10 & 0.98 & 7.80 & 43.91 & 267.61\\
\makecell[l]{Mean ECM \\ time}               & 0.07 & 1.01 & 14.15 & 141.78 & 1549.16\\
\makecell[l]{Median ECM \\ time}             & 0.05 & 0.65 & 11.05 & 102.76 & 832.52\\
\makecell[l]{Number of \\ $n$'s factored}    & 100 & 100 & 100 & 100 & 50
\end{tabular}
\caption{Comparison between factorization algorithms, in seconds}
\label{comparison table}
\end{table}
\endgroup

We see that for smaller inputs the ECM is faster than our method, but for $q \geq 10^{20}$ our method becomes significantly faster.
%This agrees with the complexity comparison of paragraph \ref{ss: complexity comparisons}.
The fact that the ECM is faster for smaller inputs can be explained by the many years of optimizations that were done to improve the method.

If the reader has tried to avoid statistics as much as the author, then they might be wondering why the mean times seem to be larger than the median times.
However, this is explained by the fact that in an exponential distribution, the mean is about $1/\ln(2)$ times the median \cite{distribution_median}. For example, for $q = 10^{30}$ we have $23.04 / \ln(2) = 33.24 \approx 30.30$.
%A possible explanation for the fact that find a somewhat smaller value is that when the algorithm takes a long to factor $n$, the value of $s$. Therefore, the size of $C(-4bs)$ increases as well, so, each round on average it gets slightly more difficult to factor $n$.

Our method that picks the multipliers $s$ is quite straightforward, start at $s=1$ and after each round go to the next square-free value.
In Appendix \ref{appendix: multipliers} we present heuristic arguments and numerical results that suggest that there are other strategies to choose your multipliers $s$ that on average improve the runtime of the algorithm.
This improvement probably won't change the complexity of the algorithm, but it can make a difference in practice.

For the remainder of this section, we place our focus on our complete algorithm.
In Table \ref{timing table}, the numbers $n$ that were factored were of the same form as in Table \ref{comparison table}. We see that if $n$ has 150 digits, then on average we can factor it in about 3 hours.

\begingroup
\renewcommand{\arraystretch}{2.5}
\begin{table}[H]
\centering
\begin{tabular}{l|cccc}
 & $q \approx 10^{20}$ & $q \approx 10^{30}$ & $q \approx 10^{40}$ & $q \approx 10^{50}$\\ \hline
Mean time                                 & 0.72 s  & 33.63 s & 9.74 m & 174.47 m\\
Median time                               & 0.54 s  & 20.55 s & 7.20 m & 87.69 m\\
\makecell[l]{Successful in \\ stage 1}    & 27\%    & 26\%    & 20\%   & 24\%\\
\makecell[l]{Successful in \\ stage 2}    & 73\%    & 74\%    & 80\%   & 76\%\\
\makecell[l]{Mean number \\ of groups}    & 6.19    & 26.80   & 63.54  & 206.36\\
\makecell[l]{Median number \\ of groups}  & 5       & 17      & 47     & 103\\
\makecell[l]{Number of \\ $n$'s factored} & 100     & 100     & 50     & 25
\end{tabular}
\caption{Timing of Algorithm \ref{decomposition algorithm 1} extended with Algorithm \ref{algorithm generic stage 2}}
\label{timing table}
\end{table}
\endgroup

Some notable examples: when we used
\[ q = 37294202675688843722966391031920296857220275388239, \]
the algorithm finished within 43 seconds, this is because the very first group that we tried was successful (in stage 2).
On the contrary,
\[ q = 53328961473418475894520883222727806445395016777723 \]
took more than 13 hours, when finally the 965th group was successful ($s = 1581$).
Upon closer inspection, we see that
\begin{gather*}
    h = h(-4 \cdot 1581 \cdot q) = 184735851610543000235261184 = \\
    2^8 \cdot 3 \cdot 19 \cdot 113 \cdot 349 \cdot 6359 \cdot 25031 \cdot 39461 \cdot 51109.
\end{gather*}
We used $B = 229158$, which is bigger than all of the prime factors of $h$, therefore stage 1 was successful.
However, if we had a method where we could have taken $B_2 = B^2$, then we would have been done after 51 groups ($s = 82$), since
\begin{gather*}
    h(-4 \cdot 82 \cdot q) = 121744820463339475628644536 = \\
    2^4 \cdot 5 \cdot 7 \cdot 32633 \cdot 35993 \cdot 153521 \cdot 1161878987.
\end{gather*}
We see that $1161878987 > B$ and also $1161878987 > 2828306 \approx B\ln(B)$. But, $1161878987 < 52513388964 = B^2$.
If we had access to such a method, then this factorization would have taken about 45 minutes instead of 13 hours.

\newpage

\appendix

\section{Algorithm \ref{decomposition algorithm 1} for composite square factors}
\label{appendix: composite square factor}

The purpose of this appendix is to provide a version of Algorithm \ref{decomposition algorithm 1} that also works for $a,b$ both possibly composite, where $n = a^2b$.
We recommend finishing Chapter \ref{s: main algorithm} first before reading this appendix.

Let's first look at what Proposition \ref{derived from e} becomes in this setting.
The main difference is that the form $f$ can now also be derived from $e$ that lies in a larger group than $C(-4b)$.
There can be many intermediate groups between $C(-4b)$ and $C(-4a^2b)$ if $a$ has many prime factors.
In the next lemma we compose transformations from \eqref{derived forms}. We restrict ourselves to transformations with $h \neq \infty$, which will make sense in the proof of Proposition \ref{derived from e composite}.

\begin{lemma}
\label{composing transformations}
    Suppose $f \in C(-4a^2b)$ is derived from $g \in C(-4b)$ using the following transformations (see (\ref{derived forms})), starting from left to right:
    \[ \begin{pmatrix}
        p_1 & h_1\\
        0 & 1
    \end{pmatrix}
    , \dots,
    \begin{pmatrix}
        p_r & h_r\\
        0 & 1
    \end{pmatrix} \]
    where the $p_i$ are primes (not necessarily distinct) and $p_1 \cdots p_r = a$ and $0 \leq h_i \leq p_i - 1$ for all $i$. Then the transformation matrix of determinant $a$ that maps $g$ to $f$ can be written as
    \[ \begin{pmatrix}
        a & h\\
        0 & 1
    \end{pmatrix}
    \quad \text{where} \quad h = \sum_{i=1}^r h_i \prod_{j=1}^{i-1} p_j \]
    Furthermore, $0 \leq h \leq a-1$.
\end{lemma}

\begin{proof}
    This can be proven by doing induction on $r$, the number of prime factors of $a$.
\end{proof}

We can now state and prove the composite analogue of Proposition \ref{derived from e}.

\begin{prop}
\label{derived from e composite}
    Suppose we have a reduced form $f \in C(-4a^2b)$ that is derived from $e_{-4b} = (1,0,b) \in C(-4b)$ and $f \not \sim e_{-4a^2b} = (1,0,a^2b)$.
    Then there exists a maximal $b_2 = (\frac{a}{a_2})^2 b$ for some $a_2 \mid a$, such that $f$ is derived from $e_{-4b_2}$, and if $b_2 > a_2^2$, then
    \[ f = (a_2^2, 2a_2k, k^2 + b_2 ) \]
    for some $-a_2/2 \leq k \leq a_2/2$.
\end{prop}

\begin{proof}
    Let $b_2 = (\frac{a}{a_2})^2 b$ be maximal such that $f$ is derived from $e = e_{-4b_2}$.
    We know that $f$ can be derived from $e$ by repeatedly applying the formulas from \eqref{derived forms} for each prime factor of $a_2$.
    Claim: only the transformations with $h \neq \infty$ in \eqref{derived forms} are used to produce $f$ from $e$.\\
    In the chain of transformations from $e$ to $f$, suppose we have an intermediate form 
    \[ f_2 \in C(-4a_3^2b_2) = C(-4(\frac{a \cdot a_3}{a_2})^2b) \]
    which is produced from $e$ using the prime factors of $a_3 = p_1 \cdots p_r \mid a_2$, and suppose only transformations of the form
    \[ \begin{pmatrix}
        p_1 & h_1\\
        0 & 1
    \end{pmatrix}
    , \dots,
    \begin{pmatrix}
        p_r & h_r\\
        0 & 1
    \end{pmatrix} \]
    were used up to this point.
    Suppose furthermore that the next prime $p \mid a_2$ in the chain does use $h = \infty$ to produce $f_3$.
    We know from Lemma \ref{composing transformations} that the transformation from $e$ to $f_2$ is of the form \[ \begin{pmatrix}
        a_3 & h\\
        0 & 1
    \end{pmatrix}
    \quad \text{with} \quad  h = \sum_{i=1}^r h_i \prod_{j=1}^{i-1} p_j \]
    Therefore, the transformation from $e$ to $f_3$ can be written as
    \[
    \begin{pmatrix}
        a_3 & h\\
        0 & 1
    \end{pmatrix}
    \begin{pmatrix}
        1 & 0\\
        0 & p
    \end{pmatrix}
    =
    \begin{pmatrix}
        a_3 & hp\\
        0 & p
    \end{pmatrix}
    \]
    However, there is another way to produce $f_3$ from $e$.
    Namely, if we instead start with the prime $p$ and use $h = \infty$ again, then we produce 
    \[ f_2' = (1,0,p^2 b_2) = e_{-4p^2b_2}\]
    from $e$.
    Next, we produce $f_3'$ by using the following transformations for the prime factors of $a_3$:
    \[ \begin{pmatrix}
        p_1 & p \cdot h_1\\
        0 & 1
    \end{pmatrix}
    , \dots,
    \begin{pmatrix}
        p_r & p \cdot h_r\\
        0 & 1
    \end{pmatrix} \]
    where $p \cdot h_i$ can be computed modulo $p_i$ for all $i$.
    Then using Lemma \ref{composing transformations} again, we see that the transformation from $e$ to $f_3'$ is:
    \[
    \begin{pmatrix}
        1 & 0\\
        0 & p
    \end{pmatrix}
    \begin{pmatrix}
        a_3 & hp\\
        0 & 1
    \end{pmatrix}
    =
    \begin{pmatrix}
        a_3 & hp\\
        0 & p
    \end{pmatrix}
    \]
    Hence $f_3 = f_3'$. We now see that $f_3$ is also derived from $e' = e_{-4p^2b_2}$. Therefore, $f$ is also derived from $e'$.
    This contradicts the assumption that $b_2$ was maximal, hence the claim is proven.\\
    To finish the proof of the proposition, we use Lemma \ref{composing transformations} one final time, together with the claim, to see that the transformation from $e = (1,0,b_2)$ to $f$ is of the form
    \[
    \begin{pmatrix}
        a_2 & l\\
        0 & 1
    \end{pmatrix}
    \quad \text{for some} \; \; 0 \leq l \leq a_2-1.
    \]
    We are now in the same situation as Proposition \ref{derived from e}.
    By repeating the steps of that proof we see that $f$ is either equal to 
    \[ (a_2^2, 2a_2l, l^2 + b_2) \quad \text{or to} \quad (a_2^2, 2a_2(l-a_2), (l-a_2)^2 + b_2).  \qedhere \]
\end{proof}

\begin{prop}
    Algorithm \ref{decomposition algorithm 1} can be extended to work for integers $n$ having composite square factors, without increasing the asymptotic complexity.
\end{prop}

\begin{proof}
    Until line 23 of Algorithm \ref{decomposition algorithm 1}, we don't have to make any adjustments.
    In line 23, if the $x^2$ coefficient of $l$ is $a^2$, then we can continue as usual.
    But, as we saw in Proposition \ref{derived from e composite}, this coefficient can also be a divisor of $a^2$, say $a_2^2$.
    This happens when our highly composite integer $k$ not just covers the prime factors of $h(-4bs)$, but also those of $h(-4(\frac{a}{a_2})^2bs)$.\\
    To get the full factor $a$ of $n$, we can use Lenstra's original Algorithm \ref{lenstra algorithm}, together with the same $s$ and $k$, to completely factor $b_2 = (\frac{a}{a_2})^2b$.
    By reading off the square part of the factorization of $b_2$, we get $\frac{a}{a_2}$.
    Combining this with the factor $a_2$ we found earlier, we now have computed the square-free decomposition of $n = (a_2 \frac{a}{a_2})^2 b$.
    Note that in this case, the only part of $n$ that we possibly don't have the full factorization of is $a_2$.\\
    We only added a single iteration of Algorithm \ref{lenstra algorithm} to factor $b_2$.
    We know from Theorem \ref{lenstra refrase} that this only takes $\tbigo(L_b[1/2,1] \ln(b_2))$ and not just $\bigo(L_{b_2} [1/2,1])$, since we use the same prime bound $B$ and multiplier $s$ that we also applied in Algorithm \ref{decomposition algorithm 1} with input $n=a^2b$.
    Therefore, the asymptotic complexity of this algorithm is the same as Algorithm \ref{decomposition algorithm 1}.
\end{proof}

\section{Choice of multipliers $s$}
\label{appendix: multipliers}

In Algorithm \ref{decomposition algorithm 1}, the order of the multipliers $s$ that we try is fairly straightforward: start at 1 and increase each round to the next square-free number.
We hope that $h(-4bs) \approx 2\sqrt{bs}/\pi$ is smooth.
So, the smallest values of $s$ should be tried first right?
Well, numerical `evidence' and heuristic arguments suggest that this is not always the case.

One of those arguments is that the number of prime factors of the discriminant strongly influences the number of factors $2$ of the class number.
%This can be seen in the following proposition.

\begin{prop}
\label{order 2 elements}
    Let $D$ be a negative discriminant, let $r$ be the number of distinct odd primes dividing $D$. Define the number $\mu$ as follows: if $D = 1 \bmod 4$, then $\mu = r$, and if $D = 0 \bmod 4$, then $D = -4m$, where $m > 0$, and $\mu$ is determined by the following table:
    \begin{center}
        \begin{tabular}{ l|c } 
            \multicolumn{1}{c}{$n$} & \multicolumn{1}{c}{$\mu$} \\
            \hline
            $m = 3 \bmod 4$ & $r$ \\ 
            $m = 1,2 \bmod 4$ & $r+1$ \\ 
            $m = 4 \bmod 8$ & $r+1$ \\ 
            $m = 0 \bmod 8$ & $r+2$
        \end{tabular}
    \end{center}
    Then the class group $C(D)$ has exactly $2^{\mu - 1}$ elements of order $\leq 2$. 
\end{prop}

\begin{proof}
    See Proposition 3.11 of \cite{cox}.
\end{proof}

If $C(D)$ has $2^{\mu - 1}$ elements of order $\leq 2$, then $C(D)$ is divisible by $2^{\mu - 1}$.
Therefore, if $n = a^2b$ and $D = -4bs$, then we only need that $h(-4bs)/2^{\mu - 1}$ is smooth to find the square-free factorization of $n$.
If we have an $s = 1 \bmod 4$ with $\gcd(s,b) = 1$, then heuristically we might expect that $h(-4bs)$ is more likely to be smooth than $h(-4b)$ if $2^r > \sqrt{s}$, where $r$ is the number of distinct odd prime factors of $s$.
Funnily enough, this suggests that taking $s = 1$ to be your first attempt is actually quite bad, since it has no prime factors!

However, it is clear that this heuristic argument doesn't give the complete picture. For example, if $C(-4b)$ contains a group isomorphic to $\Z/4\Z$, then it is possible that $C(-4bs)$ no longer has such a subgroup, but instead contains the group $\Z/2\Z \times \Z/2\Z$.
In this case, $h(-4bs)$ has the same number of factors 2 as $h(-4b)$. So, from that perspective, the probability that it is smooth has not increased.
But, as we will see in the upcoming tables, it seems that having a lot of divisors in your multiplier $s$ is generally good.

In both tables below we constructed $40.000$ integers $n$ of the form $n = p^2q$, where $p$ and $q$ are primes which are both roughly $10^{10}$.
In Table \ref{table: q = 1 mod 4} we only took primes $q = 1 \bmod 4$, and in Table \ref{table: q = 3 mod 4} we only took primes $q = 3 \bmod 4$.
Using our complete algorithm, we tested each square-free $s \leq 30$ to see if a factorization was found.
For $s = 1$ we constructed an additional 360.000 integers $n$, for a total of $400.000$ integers to factor, to ensure the accuracy of the probability estimation for $s = 1$.

The column $s$ denotes the multiplier that was used. The column `prob' is the observed probability that the factorization was successful with the given $s$. The final column shows the ratio of the success probability of the given $s$ to the success probability for $s=1$.

\vspace{1.3 em}

\begin{minipage}[c]{0.5\textwidth}
\centering
\begin{tabular}{c|c|c}
$s$ & prob & ratio to $s = 1$\\
\hline
1 & 63.63\% & 1.000  \\
2 & 59.92\% & 0.942  \\
3 & 63.77\% & 1.002  \\
5 & 66.66\% & 1.048  \\
6 & 58.44\% & 0.919  \\
7 & 59.13\% & 0.929  \\
10 & 55.15\% & 0.867  \\
11 & 56.50\% & 0.888  \\
13 & 60.90\% & 0.957  \\
14 & 62.90\% & 0.989  \\
15 & 59.79\% & 0.940  \\
17 & 62.50\% & 0.982  \\
19 & 53.43\% & 0.840  \\
21 & 61.86\% & 0.972  \\
22 & 50.67\% & 0.796  \\
23 & 52.34\% & 0.823  \\
26 & 49.27\% & 0.774  \\
29 & 56.93\% & 0.895  \\
30 & 64.17\% & 1.009
\end{tabular}
\captionof{table}{\\ Multipliers for $q = 1 \bmod 4$}
\label{table: q = 1 mod 4}
\end{minipage}
\begin{minipage}[c]{0.5\textwidth}
\centering
\begin{tabular}{c|c|c}
$s$ & prob & ratio to $s = 1$\\
\hline
1 & 54.75\% & 1.000  \\
2 & 59.49\% & 1.087  \\
3 & 62.36\% & 1.139  \\
5 & 60.65\% & 1.108  \\
6 & 65.23\% & 1.192  \\
7 & 64.35\% & 1.175  \\
10 & 55.16\% & 1.008  \\
11 & 54.93\% & 1.003  \\
13 & 55.72\% & 1.018  \\
14 & 60.41\% & 1.103  \\
15 & 66.54\% & 1.215  \\
17 & 54.12\% & 0.989  \\
19 & 51.11\% & 0.934  \\
21 & 64.99\% & 1.187  \\
22 & 58.45\% & 1.068  \\
23 & 58.05\% & 1.060  \\
26 & 50.07\% & 0.915  \\
29 & 51.58\% & 0.942  \\
30 & 62.86\% & 1.148
\end{tabular}
\captionof{table}{\\ Multipliers for $q = 3 \bmod 4$}
\label{table: q = 3 mod 4}
\end{minipage}

\vspace{1.3 em}

We see that when $q = 1 \bmod 4$, then $s=1$ is not too bad, but when $q = 3 \bmod 4$, it does not seem to be the best starting multiplier.
As we would expect from Proposition \ref{order 2 elements}, the success probability is increased when $s$ has many prime factors and when $qs = 1 \bmod 4$.
In Table \ref{table: q = 1 mod 4}, we see that $s=5$ performs better than $s=1$, but our heuristic only predicted this if $\sqrt{s} > 2^r$, which is not the case here.
Therefore, it might be extra important that $s$ has many prime factors.
For more on the two-part of the class group, see \cite{bosma_stevenhagen}.

Another way to improve the probability that $h(-4bs)$ is smooth is by taking multipliers $s$ such that $h(-4bs)$ is smaller than average.
We can try to accomplish this by using the \textit{class number formula} \cite{shanks_class_number_formula}.

\begin{theorem}
\label{class number formula}
    Let $D$ be a negative discriminant. Then
    \[ h(D) = \frac{\sqrt{-D}}{\pi} \prod_{p} \frac{p}{p-\legendre{D}{p}} \]
    where the infinite product runs over all primes $p$ and where $\legendre{D}{p}$ is the Kronecker symbol.
\end{theorem}

%Note that Proposition \ref{prime square order} is a corollary of this theorem.

We can use this formula as follows. Given an integer $n = a^2b$, we can very conveniently compute $\legendre{-4n}{p_i} = \legendre{-4b}{p_i}$ for some small primes $p_i$.
If we now choose multipliers $s$ such that for most of the symbols:
\begin{equation}
\label{symbol relation}
    \legendre{s}{p_i} = - \legendre{-4b}{p_i}
\end{equation}
then we expect $h(-4bs)$ to be relatively small on average.

Given a set of primes $p_i$, it is easy to find multipliers $s$ such that \eqref{symbol relation} holds for all $i$ by applying the Chinese remainder theorem.
But, you have to find the balance between the size of $s$ and how much you gain from using Theorem \ref{class number formula}.
This seems to be a non-trivial problem, since it also depends on the number of prime factors of $s$, as we saw in Tables \ref{table: q = 1 mod 4}, \ref{table: q = 3 mod 4}.
%Therefore, we leave the exact analysis of this problem for some other time.
We won't solve this problem in this article, but we do provide some more numerical `evidence' that suggests that this method helps in finding the square-free decomposition of $n$ faster.

The next two tables have the same structure as Tables \ref{table: q = 1 mod 4}, \ref{table: q = 3 mod 4}.
We also generated the same number of integers $n = p^2q$.
However, the condition on the primes $q$ that we consider is different.
In Table \ref{table: positive symbols} we only take primes $q$ that satisfy: $\legendre{-q}{3} = \legendre{-q}{5} = \legendre{-q}{7} = 1$.
In Table \ref{table: negative symbols} we take primes $q$ with $\legendre{-q}{3} = \legendre{-q}{5} = \legendre{-q}{7} = -1$ instead.
The new column `symbols $\legendre{-qs}{r}$' denotes the values of the three Legendre symbols for the primes $r = 3, 5, 7$.

\begin{table}[H]
\centering
    \begin{tabular}{c|c|c|c}
    $s$ & prob & ratio to $s = 1$ & symbols $\legendre{-qs}{r}$\\
    \hline
    1 & 51.58 \% & 1.000 & (1, 1, 1)  \\
    2 & 64.28 \% & 1.246 & (-1, -1, 1)  \\
    3 & 65.35 \% & 1.267 & (0, -1, -1)  \\
    5 & 74.01 \% & 1.435 & (-1, 0, -1)  \\
    6 & 62.10 \% & 1.204 & (0, 1, -1)  \\
    7 & 54.75 \% & 1.062 & (1, -1, 0)  \\
    10 & 56.69 \% & 1.099 & (1, 0, -1)  \\
    11 & 55.59 \% & 1.078 & (-1, 1, 1)  \\
    13 & 58.91 \% & 1.142 & (1, -1, -1)  \\
    14 & 58.66 \% & 1.137 & (-1, 1, 0)  \\
    15 & 58.62 \% & 1.136 & (0, 0, 1)  \\
    17 & 65.69 \% & 1.273 & (-1, -1, -1)  \\
    19 & 48.23 \% & 0.935 & (1, 1, -1)  \\
    21 & 60.16 \% & 1.166 & (0, 1, 0)  \\
    22 & 51.02 \% & 0.989 & (1, -1, 1)  \\
    23 & 59.64 \% & 1.156 & (-1, -1, 1)  \\
    26 & 53.46 \% & 1.036 & (-1, 1, -1)  \\
    29 & 54.09 \% & 1.049 & (-1, 1, 1)  \\
    30 & 59.71 \% & 1.157 & (0, 0, 1)
    \end{tabular}
    \caption{Effect of Legendre symbols on multipliers, part 1}
    \label{table: positive symbols}
\end{table}

\begin{table}[H]
\centering
    \begin{tabular}{c|c|c|c}
    $s$ & prob & ratio to $s = 1$ & symbols $\legendre{-qs}{r}$\\
    \hline
    1 & 66.68 \% & 1.000 & (-1, -1, -1)  \\
    2 & 55.27 \% & 0.829 & (1, 1, -1)  \\
    3 & 61.18 \% & 0.918 & (0, 1, 1)  \\
    5 & 53.05 \% & 0.796 & (1, 0, 1)  \\
    6 & 63.44 \% & 0.951 & (0, -1, 1)  \\
    7 & 68.29 \% & 1.024 & (-1, 1, 0)  \\
    10 & 53.63 \% & 0.804 & (-1, 0, 1)  \\
    11 & 55.58 \% & 0.834 & (1, -1, -1)  \\
    13 & 58.19 \% & 0.873 & (-1, 1, 1)  \\
    14 & 65.59 \% & 0.984 & (1, -1, 0)  \\
    15 & 61.98 \% & 0.930 & (0, 0, -1)  \\
    17 & 50.75 \% & 0.761 & (1, 1, 1)  \\
    19 & 57.06 \% & 0.856 & (-1, -1, 1)  \\
    21 & 71.64 \% & 1.074 & (0, -1, 0)  \\
    22 & 57.56 \% & 0.863 & (-1, 1, -1)  \\
    23 & 50.82 \% & 0.762 & (1, 1, -1)  \\
    26 & 46.77 \% & 0.701 & (1, -1, 1)  \\
    29 & 53.87 \% & 0.808 & (1, -1, -1)  \\
    30 & 62.84 \% & 0.942 & (0, 0, -1)
    \end{tabular}
    \caption{Effect of Legendre symbols on multipliers, part 2}
    \label{table: negative symbols}
\end{table}

It indeed seems that negative symbols imply a increased probability of finding the square-free factorization.
Also as expected, the symbols for the smallest primes $r$ are the most important, since they have the biggest impact in the formula of Theorem \ref{class number formula}.
We also again see the effect of the number of prime factors of $s$.
Whether the same patterns still hold when the size of $q$ increases has to be analyzed more thoroughly another time.

One might have noticed that we have skipped the very small prime $r = 2$. This prime (as usual) behaves a bit different than the other primes.
The even prime is also quite important, since it has the largest impact in the class number formula.
When $q = 3 \bmod 4$, it is generally favorable to have $\legendre{-q}{2} = -1$ for finding a factorization, as one might expect.
However, when $q = 1 \bmod 4$, then $-q$ is no longer a discriminant, so have to multiply it by $4$ to make it valid.
This makes $\legendre{-4q}{2} = 0$, so one might think that the value of $\legendre{-q}{2}$ has no impact on the probability of factorization.
However, in practice it seems that $\legendre{-q}{2} = -1$ is actually \emph{worse} in this case than $+1$.
Even though the number of forms of order $2$ seem roughly equal in both cases, in the former case there appears to be quite often a form of higher $2$-power order, which is beneficial for the factorization.

The idea of carefully choosing the groups that you try is again quite similar to the ECM. For example, special curves are chosen that are known to be divisible by 12 or 16 \cite{ecm_suitable_curves}, improving the probability of finding a factor.
The idea to look for small class groups does not have a direct elliptic equivalent though, since the order of an elliptic curve modulo a prime $p$ always lies in the limited (compared to class groups) interval of 
\[ \left[ p+1 - 2\sqrt{p}, \text{ } p+1 + 2\sqrt{p} \right]\] \cite{crandallprimenumbers}, Theorem 7.3.1.

In conclusion, there is very likely a way to choose your multipliers $s$ in a different order than in Algorithm \ref{decomposition algorithm 1} that performs better on average.
But, it is also clear that it requires further analysis to determine such a strategy.
Not to forget that this strategy might also change when the size of the square-free part of $n$ gets larger.
Very interesting nevertheless!

\section{Searching for a better stage 2}
\label{appendix: stage 2}

In this appendix we will discuss some methods that are used in other algebraic-group factorization algorithms where a better stage 2 is known. We will try to make one these methods work in our own algorithm.
The first one of these methods can be found in the article from Schnorr and Lenstra \cite{schnorrlenstra}.

%\subsection{Schnorr and Lenstra revisited}
\subsection{Pollard rho}
In the second stage of their factoring algorithm, Schnorr and Lenstra employ a Pollard rho type of method.
Let us briefly discuss this method and see if we can use it in our own algorithm.
In Algorithm \ref{lenstra algorithm}, we compute a form $g$ that we hope will be the identity element. If not, suppose that it has prime order $q$, where $q \in [B,B_2]$.

Define a random function $\rho : \generated{g} \longrightarrow \generated{g}$, where $\generated{g}$ is the subgroup of order $q$ in $C(-4ns)$ containing the powers of $g$.
In their article, Schnorr and Lenstra describe a few of those functions $\rho$. They all take a form $h \in \generated{g}$ and compose it with a power of $g$, depending on what the coefficients of the reduced form of $h$ are.
For example, given a reduced form $h = (a,b,c)$, define

\begin{equation}
\label{random function on forms}
    \rho(h) = \begin{cases} 
    h^2 & \text{if } a \leq \sqrt{\Delta(h)}/4,\\
    h \cdot g & \text{otherwise}.
   \end{cases}
\end{equation}

There are some more complicated mappings $\rho$ in the article of Schnorr and Lenstra that might perform better in practice (i.e. they are more random), but the idea is the same.

Given the initial form $g$, we can repeatedly apply $\rho$ to it to create a chain: $g, \rho(g), \rho(\rho(g)), \dots$.
We know from \cite{crandallprimenumbers} Section 5.2.1, that if $\rho$ is sufficiently random, then we can expect that this chain returns to a previously encountered form after $\bigo(\sqrt{q}) \subseteq \bigo(\sqrt{B_2})$ steps, after which it repeats that cycle.
Using a cycle detection algorithm, we can find such a collision in $\bigo(\sqrt{B_2})$ applications of $\rho$.
Note that the chain does not necessarily have to return to $g$.

By carefully storing the exponent of the power of $g$ that was multiplied to the initial form, the collision provides a relation of the form $g^x = g^y$, where $x \neq y$.
Then $q \mid (x-y)$, so the ambiguous forms can be constructed and $n$ can be factored.

To compute one step in the cycle, we need to compute one composition in the class group. Therefore, if we take $B_2 = B^2$, then stage 2 will take $\bigo(B)$ compositions, just like stage 1.
Lenstra \cite{schnorrlenstra} page 300, mentions that on average this provides a speedup of roughly a factor $\ln(n)/\ln\ln(n)$, compared to only using stage 1.
Summarizing, we have the following proposition.

\begin{prop}
\label{lenstra stage 2 speed-up}
    Algorithm \ref{lenstra algorithm} can be extended using a random map in $C(-4ns)$ to factor $n$, without increasing the asymptotic run time per $s$ that we try.
    It will factor $n$ if the form produced by stage 1 has order less than $B^2$.
    The number of class groups that have to be tried to factor $n$ this way is reduced by roughly a factor $\ln(n)/\ln\ln(n)$. \qed
\end{prop}

Algorithm \ref{decomposition algorithm 1} also takes place in class groups, so surely we can just use the same method to get this very nice speedup? Well, not quite.
The difference is that we don't seek a collision in the big group $C(-4ns)$, but in underlying group $C(-4bs)$ instead.
We need a random function $\rho$ that has the following property: given $f_1,f_2 \in C(-4ns)$,
\begin{equation}
\label{modulo eis}
    \pi(f_1) = \pi(f_2) \implies \pi(\rho(f_1)), \pi(\rho(f_2))
\end{equation}
where $\pi$ is the projection map $\pi : C(-4ns) \relbar\joinrel\twoheadrightarrow C(-4bs)$ from Section \ref{ss: non fun}.
This property is necessary because otherwise the cycle created by $\rho$ won't repeat fast enough in the underlying group $C(-4bs)$.
The random function in \eqref{random function on forms} does not have this property, so we can't use it. But, that doesn't mean of course that there is not another map that does tick all of the boxes.

In 1996, Okamoto and Peralta \cite{peralta_ecm_p2q} introduced a variant of the elliptic curve method (ECM) that is specialized for numbers of the form $p^2q$, where $p,q$ are primes and $q$ is relatively small.
This variant works the same as the ECM in stage 1: take some point $P$ on an elliptic curve modulo $n$ and compute $Q = [k]P$ for some large smooth $k$ and hope to find $p$ or $q$.
For stage 2, a random map is defined using the Jacobi Symbol. Note that for every $a \ne 0 \bmod p$, we have

\begin{equation}
\label{legendre mod q}
    \legendre{a}{n} = \legendre{a}{q}.
\end{equation}
From this, we can define a random map $\rho : \generated{Q} \longrightarrow \generated{Q}$ as follows:

\begin{equation}
\label{random function on points}
    \rho(R) = \begin{cases} 
    [2]R & \text{if } \legendre{R_x}{n} = 1,\\
    R + Q & \text{otherwise}.
   \end{cases}
\end{equation}

\noindent Because of \eqref{legendre mod q}, this map satisfies the elliptic equivalent of property \eqref{modulo eis}.
This means that if the order of $Q$ in $E(\F_q)$ is $r$, then the cycle $Q, \rho(Q), \rho(\rho(Q)), \dots$ will repeat in $E(\F_q)$ after $\bigo(\sqrt{r})$ steps.
Hence $B_2 = B^2$ can also be chosen in this algorithm.
If $r < B^2$ then the factor $q$ will most likely be obtained.

This approach seems very promising at first for our algorithm, because we can do something similar in class groups using \textit{assigned characters}.

\begin{definition}
    Let $D$ be a negative discriminant and let $p_1, \dots, p_r$ be the distinct odd prime factors of $D$. Then define:
    \begin{align*}
        &\chi_{p_i}(a) = \legendre{a}{p_i}, &&\text{ for $a$ prime to $p_i$, for $i = 1, \dots, r$}\\
        &\delta(a) = \legendre{-1}{a}, &&\text{ for $a$ odd}\\
        &\epsilon(a) = \legendre{2}{a}, &&\text{ for $a$ odd}
    \end{align*}
     The functions $\chi_1, \dots, \chi_r$ are always part of the assigned characters. If $D = 0 \bmod 4$, then there can be additional ones. Write $D = -4n$, then:
    \begin{center}
        \begin{tabular}{ c|c } 
            $n$ & \textup{additional assigned characters} \\
            \hline
            $n = 3 \bmod 4$ & none \\ 
            $n = 1 \bmod 4$ & $\delta$ \\ 
            $n = 2 \bmod 8$ & $\delta \cdot \epsilon$ \\ 
            $n = 6 \bmod 8$ & $\epsilon$ \\ 
            $n = 4 \bmod 8$ & $\delta$ \\ 
            $n = 0 \bmod 8$ & $\delta, \epsilon$ \\ 
        \end{tabular}
    \end{center}
    
\end{definition}

We will state two important properties of these assigned characters. These properties provide a rich structure, but at the same time they will block us from obtaining a better stage 2.

\begin{prop}
\label{character homomorphism}
    Let $D$ be a negative discriminant with assigned characters $\Psi_1, \dots, \Psi_{\mu}$, where $\mu$ is the same as in Proposition \ref{order 2 elements}.
    Let $f \in C(D)$ and suppose $f$ represents $a$ with $\gcd(D,a) = 1$. Define
    \begin{align*}
        \Phi_D : C(D&) \longrightarrow \{-1,1\}^{\mu} \\
        [f]& \longmapsto (\Psi_1(a), \dots, \Psi_{\mu}(a)).
    \end{align*}
    Then $\Phi_D$ is a well-defined homomorphism. %Hence its output is independent of the choice of $a$.
\end{prop}

\begin{proof}
    This follows from Lemma 3.13 and Lemma 3.17 from \cite{cox}.
    Note that our $\Phi_D$ is the composition (of maps) of $\Phi$ with $\Psi$ in the text of Cox.
\end{proof}

The next proposition states an important property of $\Phi$ in the case of non-fun discriminants.

\begin{prop}
\label{characters derived}
    Let $D$ be a negative discriminant with odd prime factors $p_1, \dots , p_t$ and let $r>0$ be an integer. Suppose $f_1 \in C(Dr^2)$ is derived from $f_2 \in C(D)$. Then $\Phi_{Dr^2}(f_1)$ is equal to $\Phi_D(f_2)$ for the assigned characters that they have in common.
\end{prop}

\begin{proof}
    Suppose $f_1$ represents $a \in (\Z/D\Z)^*$. Then from the formulas in \eqref{derived forms}, we see that $f_2$ also represents $a$. The result now follows from the fact that $\Phi_D$ and $\Phi_{Dr^2}$ are well-defined.
\end{proof}

% In the setting of Algorithm \ref{decomposition algorithm 1}, we have $n = a^2b$, with $b = \prod_{i = 1}^t p_i^{e_i}$, and we work in the class group $C(-4nsr^2)$, with $s = \prod_{i = 1}^u q_i^{s_i}$.
% We can always compute $\chi_{q_i}(c)$ for all $i = 1, \dots, u$ and $\chi_n(c)$, where
% \[ \chi_n(c) = \chi_b(c) = \prod_{i = 1}^t \chi_{p_i}(c)^{e_i} = \prod_{i = 1}^t \chi_{p_i}(c) \quad \textit{if } \gcd(c,a) = 1 . \]

%In the setting of Algorithm \ref{decomposition algorithm 1}, we have $n = a^2b$ and we work in the class group $C(-4nsr^2)$, with $s = \prod_{i = 1}^u q_i^{s_i}$.
%We can always compute $\chi_{q_i}(c)$ for all $c$ and $i = 1, \dots, u$, and also $\chi_n(c)$, where

In the setting of Algorithm \ref{decomposition algorithm 1}, we have $n = a^2b$ and we work in the class group $C(-4nsr^2)$.
We can compute $\chi_n(c) = \legendre{c}{n}$, where
\[ \chi_n(c) = \chi_b(c) \quad \textit{if } \gcd(c,a) = 1. \]

We can now try to use the idea from \eqref{random function on points} for our stage 2. Let $g \in C(-4nsr^2)$ be the form that we will use for stage 2.
Define $\rho : \generated{g} \longrightarrow \generated{g}$ as:
\begin{equation}
\label{random function on forms 2}
    \rho(h) = \begin{cases} 
    h^2 & \text{if } \chi_n(h) = 1,\\
    h \cdot g & \text{otherwise}.
   \end{cases}
\end{equation}

Here $\chi_n(h)$ means $\chi_n(c)$, where $h$ represents $c \in (\Z/D\Z)^*$.
Because of $\chi_n(c) = \chi_b(c)$ and Propositions \ref{derived behaves well} and \ref{characters derived}, the map $\rho$ very nicely meets the requirement \eqref{modulo eis}.
However, because of Proposition \ref{character homomorphism}, the map is no longer random, since
\[ \chi_n(h^2) = 1 \text{ and } \chi_n(h \cdot g) = \chi_n(h) \cdot \chi_n(g) \]
This makes the `random' map $\rho$ very predictable.
Therefore, $\rho$ won't become periodic quickly enough in $C(-4bs)$.
Every variation of the map $\rho$ that depends on the function $\chi_n$ and which uses the standard composition of forms runs into this problem. 
This problem doesn't exist in \eqref{random function on points}, because there doesn't seem to be an obvious relation between $\legendre{R_x}{n}$ and $\legendre{\rho(R_x)}{n}$.
But again, there could still be a random map that does work for our setting as well.

%\subsection{Elliptic curve method (ECM)}
%\subsection{Baby-steps giant-steps}
\subsection{Fast Fourier transform}
Another possible way to get a better stage 2 is by applying the fast Fourier transform (FFT), also known as magic. The FFT has many applications, one of them is fast evaluation of polynomials.

Given a polynomial $f \in R[x]$ of degree $m$, where $R$ is a ring where the FFT is possible, such as $\Z/n\Z$.
Let $a \in R$, we can compute $f(a)$ in $\bigo(m)$ multiplications and additions in $R$ by computing $a, a^2, \dots, a^m$ and then taking the appropriate sums.
So, if we are given $m$ elements $a_1, \dots, a_m \in R$ that are completely arbitrary, then surely we need at least $\bigo(m^2)$ operations in $R$ to compute the evaluations $f(a_1), \dots, f(a_m)$. Right...?

\begin{theorem}
\label{fft evaluations}
    $f(a_1), \dots, f(a_m)$ can be computed in $\bigo(m\ln(m)^2)$ operations in $R$.
\end{theorem}

\begin{proof}
    The proof uses the fact that 
    \[ f(a_i) = f(x) \bmod (x-a_i). \]
    Computing these reductions directly takes too much time because the degree of $f$ can be large.
    It turns out that gradually reducing $f$ modulo intermediate products of the linear polynomials $(x-a_i)$ before computing $f(x) \bmod (x-a_i)$ is faster.\\
    See Section 9.6.3 of \cite{crandallprimenumbers} for an actual proof. 
\end{proof}

A common stage 2 continuation of the ECM uses Theorem \ref{fft evaluations} to factor an integer $n$ as follows. Suppose again that $Q = [k]P$ is a point on an elliptic curve $E$ modulo $n$ that we got from stage 1. Suppose furthermore that $p \mid n$ is prime and the order of $Q \in E(\F_p)$ is $q < B_2$, for some prime $q$.

Define $t = \lceil \sqrt{B_2} \rceil$, now compute two lists that contain the following multiples of $Q$ (all calculations are done in $E(\Z/n\Z)$):
\begin{equation}
\label{FFT lists}
    L_1 = \{ \bigo, Q, [2]Q, [3]Q, \dots, [t-1]Q \}, \quad L_2 = \{ [t]Q, [2t]Q, \dots, [t^2]B \}.
\end{equation}

\noindent Now, $q < B_2$, hence there exists $0 \leq i,j-1 \leq t-1$ such that $q = i \cdot t - j$. Thus,
\[ [it]Q = [j]Q \text{ when viewed in } E(\F_p), \text{ where } [j]Q \in L_1 \text{ and } [it]Q \in L_2.\]
%We will find these points using Theorem \ref{fft evaluations}.
%Afterwards, we will use the fact that the $x$-coordinates of the points are equal mod $p$.
The difference of the $x$-coordinates of $[it]Q$ and $[j]Q$ will provide the factor $p$.

Construct 
\[ f = \prod_{R \in L_1} (x - R_x) \in \Z/n\Z[x]. \]
Then evaluate this polynomial using Theorem \ref{fft evaluations} in all the points $\{ R_x \mid R \in L_2 \}$ to obtain the values $b_1, \dots, b_t \in \Z/n\Z$.
Compute $\gcd(n,b_i)$ for $i = 1, \dots, t$ to find the factor $p$ of $n$.
In total this takes $\bigo(\sqrt{B_2}\ln(B_2)^2) = \bigo(B\ln(B)^2)$ operations mod $n$ if we take $B_2 = B^2$. This is slightly more than the $\bigo(B)$ operations mod $n$ needed in stage 1 of the ECM, so we need to take $B_2$ somewhat smaller, namely $B^2/\ln(B)^4$.

The choice of the sets $L_1$ and $L_2$ can be further optimized like in \cite{montgomeryECMFFT} Chapter 5, but this is the general idea.
Summarizing, we have the following proposition.

\begin{prop}
    The ECM can be extended using the FFT to factor $n$ if the point produced by stage 1 has order less than $B^2/\ln(B)^4$ in the underlying group, without increasing the asymptotic run time per curve that we try.
    %About $\ln(n)/\ln\ln(n)$ fewer curves have to be tried to factor $n$ this way.
    \qed
\end{prop}

Unfortunately, when we have two forms $f_i = (a_i,b_i,c_i) \in C(-4ns)$, $i = 1,2$ that are derived from the same form in $C(-4bs)$, then we can't expect a relation like $a_1 = a_2 \bmod b$ to hold, where $n = a^2b$.
The values $a_1,a_2 \bmod b$ heavily depend on the choice of representatives.
Therefore we can't use the coefficients directly to determine whether two forms are derived from the same form.
We could check if $f_1 \cdot f_2^{-1}$ is derived from $e$ using the idea from Proposition \ref{derived from e}.
Unfortunately, the process of computing the product $f_1 \cdot f_2^{-1}$ and then reducing it can't be written as a polynomial function (as far as we know), therefore we can't use it in the FFT method.

But, perhaps we can still use the FFT method with another property of the forms. One promising idea introduces an additional factor $r^2$ again, like in Lemma \ref{r truc}.
The main idea is that the product $a_1a_2$ will be represented by $e$ in the underlying group.
%We need the factor $r^2$ for the same reason as before, namely when the square part of $n = a^2b$ is not large enough.
We would like to stress that this factor $r$ does not necessarily have to be prime. This flexibility could be quite useful, but taking it prime makes the presentation easier.

\begin{lemma}
\label{e represents a1a2}
    Suppose $f_1,f_2 \in C(-4n)$ are derived from the same form in $C(-4b)$. Let $r$ be a prime. Lift $f_1,f_2$ to forms $g_1,g_2 \in C(-4nr^2)$ and let $h_1 = g_1^{\varphi_{-4n}(r)}$ and $h_2 = g_2^{\varphi_{-4n}(r)}$.
    Furthermore, suppose that $h_i \sim (a_i,b_i,c_i)$, $i = 1,2$.
    Then the form $e_{-4br^2}(x,y) = x^2 + br^2y^2$ represents $a_1a_2$.
\end{lemma}

\begin{proof}
    Using Lemma \ref{r truc}, we see that $h_1 h_2^{-1}$ is derived from $e_{-4br^2}$.
    Now, $h_1$ represents $a_1$, and $h_2^{-1}$ represents $a_2$.
    Therefore, $h_1 h_2^{-1}$ represents $a_1a_2$, hence $e_{-4br^2}$ also represents $a_1a_2$.
    %so the composition algorithm of forms provides us $x,y \in \Z$ such that $e_{-4br^2}(x,y) = a_1a_2$.
    %The composition algorithm provides a representation of $a_1a_2$ by $e_{-4br^2}$.
\end{proof}

Because of this lemma, we can instead think about what type of integers are represented by $e_{-4br^2}$ to determine which forms in $C(-4n)$ are derived from the same form.
The next proposition tells us a lot about the value of $x$ in the representation $e_{-4br^2}(x,y) = a_1a_2$, without the knowledge of $b$.
The statement is quite technical, so we will look at an example before proving it.

\begin{prop}
\label{compute x}
    Assume the assumptions from Lemma \ref{e represents a1a2}. Furthermore, suppose that $h_1,h_2$ are reduced, $r > \frac{4}{3}\sqrt{n}$ and $\gcd(r,a_1a_2) = 1$.
    Let $x_i = \sqrt{a_i} \bmod r^2$, for $i = 1,2$ and let $x_3 = \pm x_1 x_2 \bmod r^2$, where the sign is chosen such that the lift of $x_3$ to $x_4 \in \Z$ is as small as possible in absolute value.
    Then $x_4 = \pm x$, where $x$ is from the representation $e_{-4br^2}(x,y) = a_1a_2$.
\end{prop}

\begin{example}
    \normalfont
    Let $n = 37559 = 23^2 \cdot 71$ and $r = 37273$, a prime of roughly size $n$. Let's suppose that we don't know $a = 23$ and $b = 71$, but that we do have forms $f_1 = (125,-42,304), f_2 = (40,-18,941) \in C(-4n)$ that are both derived from $(8,-2,9) \in C(-4b)$. Following Lemma \ref{e represents a1a2}, we compute 
    \[ h_1 = (5981917,450638,8731416), h_2 = (1138784,754166,45945525) \in C(-4nr^2). \]
    This is independent of the choice of lifts $g_1$ and $g_2$.
    Afterwards, we compute
    \[ x_1 = \sqrt{5981917} = 827751348 \bmod r^2, \; x_2 = \sqrt{1138784} = 467938638 \bmod r^2. \]
    Their product is $x_1x_2 = x_3 = 1386685492 \bmod r^2$, but $-x_3 = 2591037 \bmod r^2$, which is smaller.
    Therefore we have $x_4 = 2591037 \in \Z$.\\
    It can be checked that
    \[ e_{-4br^2}(x_4,1) = x_4^2 + br^2 = a_1a_2. \]
    This also means that $\gcd(x_4^2 - a_1a_2, n)$ would reveal the factor $b$ of $n$.
    Also note that $x_4$ is relatively small compared to $r^2 = 1389276529$.
\end{example}

Let's now do the proof of Proposition \ref{compute x}.

\begin{proof}
    % The first step is also true when $r$ is not prime
    First note that $x_1$ and $x_2$ exist mod $r^2$ since $h_1$ and $h_2$ are squares.
    Now, $x_4^2 = a_1a_2 = x^2 \bmod r^2$, hence $x_4 = \pm x \bmod r^2$.
    Therefore, we can write $x = \pm x_4 + kr^2$, for some integer $k$.
    We know that $x^2 + br^2y^2 = a_1a_2$, hence
    \[ (\pm x_4 + kr^2)^2 + br^2y^2 = a_1a_2. \]
    Now, $h_1,h_2 \in C(-4nr^2)$ are reduced, so $a_1, a_2 \leq \sqrt{4nr^2}/3$, hence $a_1a_2 \leq 4nr^2/9$.
    Also, $x_4$ was chosen as small as possible, so if $k \neq 0$, then
    \[a_1a_2 = (\pm x_4 + kr^2)^2 + br^2y^2 \geq (\pm x_4 + kr^2)^2 \geq (r^2/2)^2 = r^4/4 > 4nr^2/9 \geq a_1a_2. \]
    Thus $k = 0$, meaning $x = \pm x_4$.
\end{proof}

We can also see that if $r$ is large, say roughly the size of $n$, then 
\[ x_4 \leq \sqrt{a_1a_2} \leq \sqrt{4nr^2}/3 \approx r^{3/2}. \]
%This means that if we have a list of forms $f_1, \dots, f_B \in C(-4n)$, then we can identify the forms that are derived from the same form in $C(-4b)$ by checking whether the product $a_1a_2 \bmod r^2$ is small.
This means that given two forms $f_1, f_2 \in C(-4n)$, a necessary condition that they are derived from the same form in $C(-4b)$ is that the product $x_1x_2 \bmod r^2$ is small in absolute value, where $x_1, x_2$ and $r$ are as in Proposition \ref{compute x}.

The problem is that recognizing whether a number mod $r^2$ is small or not is also not a polynomial expression (as far as we know).
Otherwise, we could compute two lists as in \eqref{FFT lists}, containing the forms $f_i = (a_i,b_i,c_i)$, which are powers of the form produced in stage 1.
Then compute the square roots $x_i = \sqrt{a_i} \bmod r^2$, and use the FFT to find $x_i$ and $x_j$ such that $x_ix_j \bmod r^2$ is small.
Then $f_i f_j^{-1}$ would likely reveal the factor $a$ of $n$.
%Otherwise we could use the FFT method with the $a_i \bmod r^2$ to find two forms $g_1, g_2$ that are derived from the same form. Then $g_1 g_2^{-1}$ would reveal the factor $a$ of $n$.

One other attempt to use the FFT uses that $\gcd(x_4^2 - a_1a_2, n) = b$ if $f_1$ and $f_2$ are derived from the same form (and $f_1 \neq f_2$). 
One of the problems here is that $x_4 = x_1x_2 \bmod r^2$ first has to be fully reduced mod $r^2$.
However, the gcd is with $n$ and not with $r^2$ again.
While creating the FFT polynomial, we can either work modulo $n$ or modulo $r^2$, but not both.
%However, there are four variables at play here: $x_1,x_2$ and $a_1,a_2$.
%This leads to multiple problems.
%One of them is that this would make the FFT polynomial bivariate. In this case we can't expect to evaluate this polynomial as fast as in Theorem \ref{fft evaluations}.

In summary, we were not able to find a better stage 2 for our algorithm.
Fortunately, this is by no means an exhaustive list of possible approaches for a better stage 2. Hopefully the examples given in this appendix will provide some inspiration.

\section{Implementation in \textsc{Magma}}
\label{appendix: implementation}

We implemented a basic version of Algorithm \ref{decomposition algorithm 1}, together with the stage 2 extension as presented in Algorithm \ref{algorithm generic stage 2}.
We specialized it for numbers $n$ of the form $n=p^2q$, where $p,q$ are primes of roughly the same size.
We used this code for all the tables that are present in this article.
This code should only be used as a start, it is definitely not optimized.
We also did not include a lot of error handling, nor the construction of ambiguous forms in line 17 of Algorithm \ref{decomposition algorithm 1}, since it is very unlikely to happen for integers $n$ of this form anyways, especially when $n$ is large.

\begin{verbatim}
SquareFactorization := function(n)  // n = p^2*q

    q2 := n^(1/3);  // we assume that q is roughly equal to p
    e := Sqrt(Log(q2)/Log(Log(q2)));
    B := 2*Round(q2^(1/(2*e)));  // we took B somewhat bigger,
    // not sure if it is better

    prod := 1;  // prod is the integer k
    for p in PrimesUpTo(B) do
        prod *:= p^(Ceiling(Log(p,B)));
    end for;

    B2 := 2*B*Round(Log(B));
    stage2primes := PrimesInInterval(B,B2);

    largeprime := NextPrime(Round(10*n^(1/6)));
    // this is the prime r from Algorithm 2.
    // if p/q <= 10, then this r is large enough

    s := 0;

    while true do
        s +:= 1;
        if IsSquarefree(s) eq false then
            continue;
        end if;

        // STAGE 1
        m := n*s;
        Q := QuadraticForms(-4*m);
        Q2 := QuadraticForms(-4*m*largeprime^2);

        for r in PrimesInInterval(3,1000) do
        // we search for an element in C(-4*m)
        // we expect this to work within a few tries
            if LegendreSymbol(-m,r) ne 1 then
                continue;
            end if;
            b := Integers() ! Sqrt(Integers(4*r) ! (-4*m));
            f := Q ! [r,b,(4*m+b^2) div (4*r)];
            break;
        end for;

        g := f^prod;  // this is the main part of stage 1
        // we hope that g is the identity element in 
        // the underlying group C(-4*q*s)

        h := Q2 ! [g[1]*largeprime^2, g[2]*largeprime, g[3]];
        // we lift g to the larger group C(-4mr^2)
        // this makes it possible to retrieve p^2
        h := Reduction(h);
        k := h^(largeprime - LegendreSymbol(-m,largeprime));

        d := GCD(k[1],n);
        if d gt 1 and d lt n then  // if a factor is found
            return(d);  // then return it
        end if;
        
        // STAGE 2
        k2 := k^2;
        smallpowers := [k2];
        k3 := k2;
        smallbound := Round(10*Log(B2));  // largest prime gap
        // for our tables this bound was large enough
        for i in [1..smallbound] do
            k2 *:= k3;
            Append(~smallpowers, k2);
        end for;

        previousprime := 0;
        // for each prime z in [B,B_2] we will consider k^z
        for p1 in stage2primes do
            currentprime := p1;
            if previousprime eq 0 then
                l := k^currentprime;
            else
                smallstep := (currentprime - previousprime) div 2;
                if smallstep gt smallbound then
                    "help", smallbound, smallstep;
                    break;
                else
                    // we try the next prime
                    l *:= smallpowers[smallstep];
                end if;
           end if;

            d := GCD(l[1],n);
            if d gt 1 and d lt n then
                return(d);
            end if;
            previousprime := currentprime;
        end for;

        // if stage 2 was also unsuccessful,
        // then take the next s and try again

    end while;

end function;
\end{verbatim}

\addcontentsline{toc}{section}{References}

\bibliographystyle{plainurl}
\bibliography{bibliography.bib}

\end{document}